\documentclass[12pt]{amsart}

\usepackage{graphicx}
\usepackage{pst-all}
\usepackage{amsmath}
\usepackage{amsfonts}
\usepackage{amssymb}
\usepackage{amsthm}
\usepackage{url}
\usepackage{subfigure}
\input xy
\xyoption{all}
\usepackage{nicefrac}

\theoremstyle{theorem}
\newtheorem{theorem}{Theorem}
\newtheorem{corollary}[theorem]{Corollary}
\newtheorem{lemma}[theorem]{Lemma}
\newtheorem{proposition}[theorem]{Proposition}
\theoremstyle{definition}
\newtheorem{definition}[theorem]{Definition}
\newtheorem{example}[theorem]{Example}
\newtheorem{remark}[theorem]{Remark}
\newtheorem{notation}[theorem]{Notation}

\newenvironment{thmbis}[1]
  {\renewcommand{\thetheorem}{\ref{#1}$^\prime$}%
   \addtocounter{theorem}{-1}%
   \begin{theorem}}
  {\end{theorem}}

\thanks{The author is partially supported by the Spanish Ministerio de Econom{\'{i}}a y Competitividad (grant MTM 2012-30719).}
\address{Facultad de Ciencias Econ{\'{o}}micas y Empresariales. Universidad Aut{\'{o}}noma de Madrid, Campus Universitario de Cantoblanco. 28049 Madrid (Espa{\~{n}}a)}
\email{JaimeJ.Sanchez@uam.es}

\begin{document}

\author{J. J. S\'anchez-Gabites}
\title{On the set of wild points of attracting surfaces in $\mathbb{R}^3$}
\begin{abstract} Suppose that a closed surface $S \subseteq \mathbb{R}^3$ is an attractor, not necessarily global, for a discrete dynamical system. Assuming that its set of wild points $W$ is totally disconnected, we prove that (up to an ambient homeomorphism) it has to be contained in a straight line. Using this result and a modification of the classical construction of a wild sphere due to Antoine we show that there exist uncountably many different $2$--spheres in $\mathbb{R}^3$ none of which can be realized as an attractor for a homeomorphism.
\end{abstract}

\maketitle

\section*{Introduction}

It is well known that attractors of dynamical systems usually have a very intricate structure as subsets of the phase space, and this prompts naturally the somewhat vague problem of trying to understand how complicated they can be. More specifically, one can consider the following \emph{realizability problem}: given a compact set $K$, can it be realized as an attractor for a dynamical system on some $\mathbb{R}^n$? The answer depends on the dimension $n$ of the ambient space, on whether $K$ is given in the abstract or already as a subset of $\mathbb{R}^n$, on whether one considers global or local attractors, continuous or discrete dynamics, etc. There are numerous papers in the literature that deal with this question in some variant or another, among which we may cite \cite{crovisierrams1}, \cite{duvallhusch1}, \cite{garay1}, \cite{gunther1}, \cite{gunthersegal1}, \cite{jimenezllibre1}, \cite{peraltajimenez1}, \cite{kato1}, \cite{mio5} or \cite{sanjurjo1}. Here we shall concentrate on the realizability problem for closed surfaces $S \subseteq \mathbb{R}^3$ as attractors for homeomorphisms. This case appears naturally in dynamics; for instance, one may think of the invariant tori that may arise in an action-angle description of a Hamiltonian system.

In previous papers we were able to prove that the horned sphere of Alexander cannot be realized as an attractor \cite{mio6} and, jointly with R. Ortega \cite{ortegayo1}, that the same holds true for the sphere of Antoine. Building on this it was possible to construct surfaces of any genus that cannot be attractors either. All these examples are wild surfaces in the sense of geometric topology; roughly speaking, this means that they cannot be smoothed out within $\mathbb{R}^3$ (the formal definition will be given in the next section). This is no accident, since a surface that is not wild can always be realized as an attractor (Proposition \ref{prop:tame_att}), and it leads one to wonder what the precise relation between the wildness of a surface and the possibility of realizing it as an attractor could be. In this paper we obtain a result in this direction showing that, if the set $W$ of wild points of an attracting surface $S$ is totally disconnected, then it must be rectifiable; that is, it must be possible to perform an ambient homeomorphism of $\mathbb{R}^3$ that sends $W$ into a straight line (Theorem \ref{teo:main2}).

In the particular case of the sphere of Antoine mentioned earlier, the set of wild points is a rather peculiar Cantor set called ``Antoine's necklace'' which is not rectifiable, so this sphere cannot be an attractor. This result was originally proved in \cite{ortegayo1} by a careful analysis of the fundamental group of the complement of Antoine's necklace which depended crucially on the hierarchical nature of the construction of the necklace and therefore was very specific to that particular example. The approach of the present paper is more general, in that it only involves the rectifiability of the set of wild points rather than any of its finer details. This generality will allow us to prove, in particular, that there are uncountably many different ways to embed any given surface in $\mathbb{R}^3$ in such a way that it cannot be realized as an attractor.

\section{Statement of results} \label{sec:statement}

Let $f$ be a homeomorphism of $\mathbb{R}^3$. By an \emph{attractor} $K$ for $f$ we mean a compact invariant set (that is, $f(K) = K$) which attracts all compact subsets of some neighbourhood $U$ of $K$. This means that for every compact set $P \subseteq U$ and every neighbourhood $V$ of $K$ in $\mathbb{R}^3$ there exists $n_0 \in \mathbb{N}$ such that $f^n(P) \subseteq V$ for every $n \geq n_0$. In the language of dynamics, $K$ is stable in the sense of Lyapunov and attracts all points in $U$. Notice that no particular assumption is made about the size of $U$, so $K$ is a local but not necessarily global attractor.

By a \emph{closed surface} we mean a compact, boundariless $2$--manifold. A closed surface $S \subseteq \mathbb{R}^3$ is \emph{locally flat} at a point $p$ if $p$ has a neighbourhood $U$ in $\mathbb{R}^3$ such that the pair $(U,U\cap S)$ is homeomorphic to $(\mathbb{R}^3, \mathbb{R}^2 \times \{0\})$. This definition is standard; see for instance the book by Daverman and Venema \cite[p. xvi]{davermanvenema1}. Intuitively it means that, near $p$, the surface lies in $\mathbb{R}^3$ like a plane. For instance, every piecewise linear surface $S \subseteq \mathbb{R}^3$ is locally flat at every point, and the same holds true when $S$ is a differentiable submanifold of $\mathbb{R}^3$ because of the existence of adapted charts.

Our first result is very simple:

\begin{proposition} \label{prop:tame_att} Let $S \subseteq \mathbb{R}^3$ be a closed surface that is locally flat at every point. Then $S$ is an attractor for a homeomorphism (and actually, also for a flow) of $\mathbb{R}^3$.
\end{proposition}

The only purpose of the above proposition is to show that, since we are interested in finding surfaces that cannot be attractors, we need to turn our attention to surfaces that contain points at which the surface is not locally flat. These points, and the surface itself, we call \emph{wild}, although this is a slight abuse of terminology\footnote{Formally, a wild point is one where the surface is not locally \emph{tame} rather than not locally flat. However, in the case that concerns us now, both notions are equivalent. More on this in Section \ref{sec:prop1}.}. The set of wild points is clearly closed so, when nonempty, it is compact. It might be difficult at first to imagine how a surface could fail to be locally flat, but we shall see an example after stating our main theorem below.

We need a final definition: let us say that a totally disconnected compact set $T \subseteq \mathbb{R}^3$ is \emph{rectifiable} if there exists a homeomorphism of $\mathbb{R}^3$ that sends $T$ into a straight line. Again, it might be difficult to imagine how a totally disconnected compact set $T \subseteq \mathbb{R}^3$ may not be rectifiable. Rather than presenting an example now, we postpone this to Section \ref{sec:cantor} where we recall the definition of ``Antoine's necklace'', which is a Cantor set that is not rectifiable. As the reader will see, the construction is very flexible and can be modified to produce a great variety of Cantor sets that are not rectifiable.

The main result of the paper is the following:

\begin{theorem} \label{teo:main2} Let $S \subseteq \mathbb{R}^3$ be a closed, connected surface that bounds a $3$--manifold. Suppose that $S$ contains a compact, totally disconnected set $T$ such that:
\begin{itemize}
	\item[(i)] $S$ is locally flat at each $p \not\in T$.
	\item[(ii)] $T$ is not rectifiable.
\end{itemize}
Then $S$ cannot be realized as an attractor for a homeomorphism of $\mathbb{R}^3$.
\end{theorem}

Theorem \ref{teo:main2} is stated in the form most convenient for the construction of nonattracting surfaces that follows below. However, it can also be presented (as we did in the abstract and the Introduction) as an assertion about the set of wild points of an attracting surface, as follows:

\begin{thmbis}{teo:main2} \label{teo:main2prima} Let $S \subseteq \mathbb{R}^3$ be a closed surface that bounds a $3$--manifold and is an attractor for a homeomorphism. Suppose that its set of wild points $W$ is totally disconnected. Then $W$ must be rectifiable.
\end{thmbis}

The condition that $S$ bounds a $3$--manifold means the following. Recall that any closed, connected surface $S$ separates $\mathbb{R}^3$ into two connected components $U_i$ whose closures are $\bar{U}_i = U_i \cup S$. This follows from Alexander duality and the fact that $S$ is a surface (see for instance the argument in \cite[Theorem 36.3, p. 205]{munkres3}). We say that $S$ \emph{bounds a $3$--manifold} if at least one of the $\bar{U}_i$ is a $3$--manifold (possibly noncompact) with boundary. The definition for a non connected $S$ is similar. Although this condition is fulfilled in many of the classical examples of wild surfaces, it is added for technical reasons and it would be nice if it could be removed from the theorem.
\medskip

\underline{\it Constructing nonattracting surfaces}. We now explain how to construct examples of surfaces that cannot be attractors. The procedure we describe is essentially the same as the one used by Antoine to construct a wild sphere. Let $T \subseteq \mathbb{R}^3$ be a totally disconnected, compact set of $\mathbb{R}^3$. Eventually we shall take $T$ to be non rectifiable, but this is not necessary for the moment. As with any compact subset of $\mathbb{R}^3$, we can find a sequence of compact $3$--manifolds with boundary $(N_k)_{k=1}^{\infty}$ such that $N_{k+1}$ is contained in the interior of $N_k$ for each $k$ and $T$ is the intersection of the $N_k$. Since $T$ is totally disconnected, in addition we may require that (i) the diameters of the connected components of the $N_k$ converge to zero as $k \rightarrow \infty$ and (ii) every connected component of each $N_k$ meets $T$, since those that do not can be removed. Figure \ref{fig:tot_disc} shows a totally disconnected set $T$ (suggested by the cloud of dots, presumably infinite) and the two first neighbourhoods of a possible sequence $(N_k)$. The first one is a single cell. The second one $N_2$ already consists of three connected components $N_{2,1}$, $N_{2,2}$, $N_{2,3}$ (a cell, a solid torus, and a solid double torus).

\begin{figure}[h]
\null\hfill
\subfigure[]{
\begin{pspicture}(0,0)(4.5,4.5)
\rput[bl](0,0){\scalebox{0.75}{\includegraphics{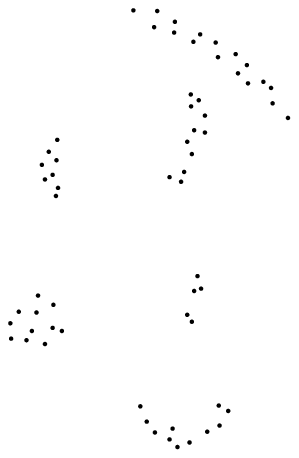}}}
\rput[bl](3.4,2.5){$T$}
\end{pspicture}}
\hfill
\subfigure[]{
\begin{pspicture}(0,0)(4.5,4.5)
\rput[bl](0,0){\scalebox{0.75}{\includegraphics{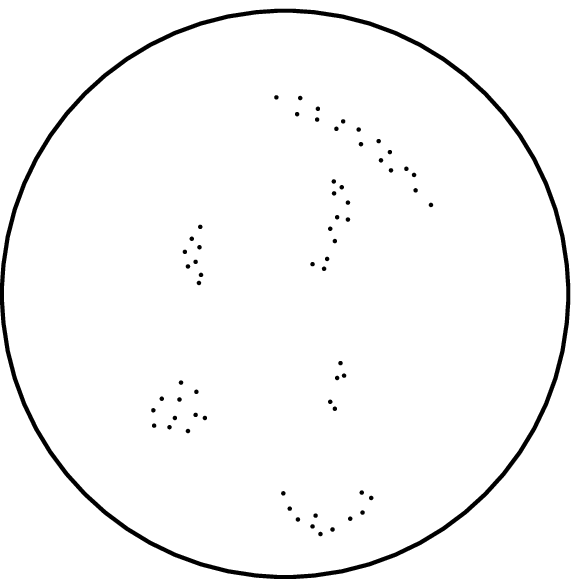}}}
\rput[bl](4,0.5){$N_1$}
\end{pspicture}}
\hfill
\subfigure[]{
\begin{pspicture}(0,0)(4.5,4.5)
\rput[bl](0,0){\scalebox{0.75}{\includegraphics{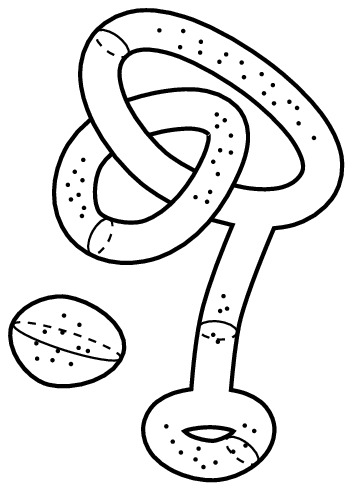}}}
\rput[bl](3,1.5){$N_2$}
\end{pspicture}}
\hfill\null
\caption{A totally disconnected set $T$ and a couple of its neighbourhoods $N_k$ \label{fig:tot_disc}}
\end{figure}

Let $S_0 \subseteq \mathbb{R}^3$ be a closed, locally flat surface such that $T$ lies in its exterior. For simplicity we shall take $S_0$ to be the boundary of a closed, round ball $B_0$ disjoint from $T$. We are going to modify $B_0$, without changing its topological type, so that the resulting $3$--manifold $B$ contains $T$ in its boundary. In order to do this we take the $N_k$ as a guide. First, we extract a solid ``feeler'' from $B_0$ that connects it to the boundary of $N_1$, meeting it in a disk $D_1$. See Figure \ref{fig:modify_B}.(a), where $B_0$ and the feeler are shown in a thicker outline and $D_1$ is shaded gray. Then, within $N_1$, we branch the feeler into ``subfeelers'', each one connected to a different component $N_{2,j}$ of $N_2$. Again, these subfeelers should meet the boundary of their corresponding $N_{2,j}$ in a disk $D_{2,j}$. See Figure \ref{fig:modify_B}.(b). Although Figure \ref{fig:modify_B} does not show this for simplicity, the feelers may be knotted or entangled.

\begin{figure}[h]
\null\hfill
\subfigure[]{
\begin{pspicture}(0,0)(8,4.5)
\rput[bl](0,0){\scalebox{0.75}{\includegraphics{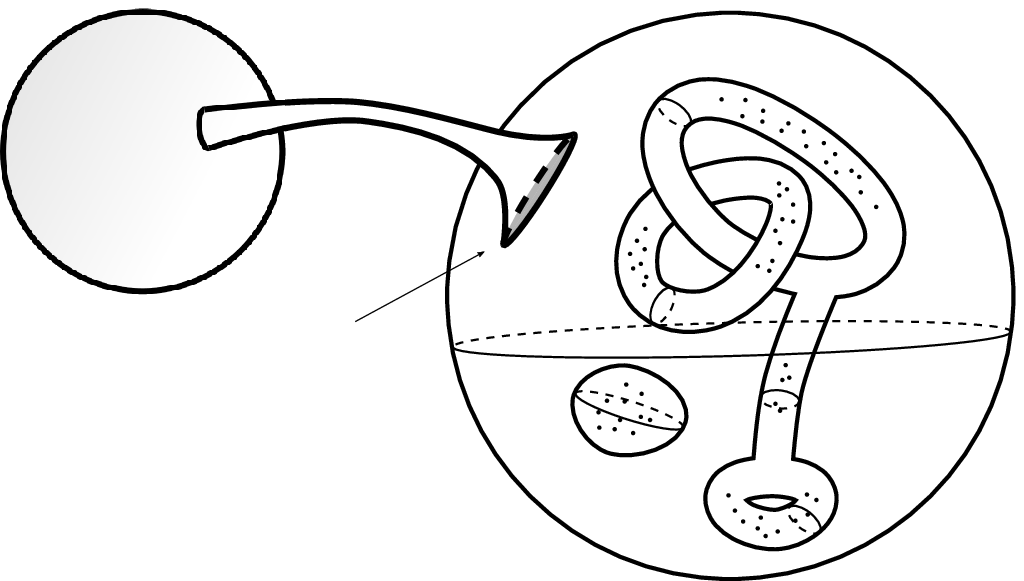}}}
\rput[bl](1.9,2.2){$B_0$} \rput[bl](7.4,0.4){$N_1$}
\rput[tr](2.8,2){$D_1$}
\end{pspicture}}
\hfill
\subfigure[]{
\begin{pspicture}(0,0)(8,4.5)
\rput[bl](0,0){\scalebox{0.75}{\includegraphics{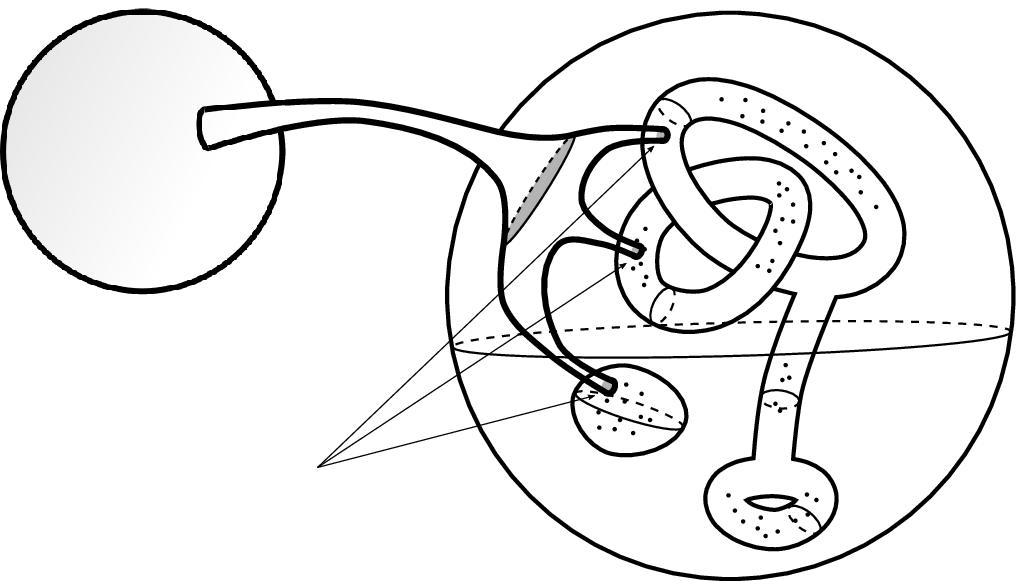}}}
\rput[bl](1.9,2.2){$B_0$}
\rput[bl](4.4,0.5){$N_{2,1}$} \rput[bl](6.2,1){$N_{2,3}$}
\rput[tr](2.8,0.8){$D_{2,j}$}
\end{pspicture}}
\hfill\null
\caption{Modifying $B_0$ so that its boundary contains $T$ \label{fig:modify_B}}
\end{figure}

We work in the same fashion within each component $N_{2,j}$ of $N_2$, splitting the subfeelers into even thinner subfeelers connected to those components of $N_3$ contained in $N_{2,j}$; then at an even smaller scale within the components $N_{3,j}$ of $N_3$, and so on. This gives rise to a nested sequence of compact $3$--manifolds, each larger but homeomorphic to the previous one. Let $B$ be the union of this sequence and the set $T$, which is the limit set of the feelers. It is not difficult to check that $B$ is a compact $3$--manifold homeomorphic to $B_0$ (see for instance the arguments in \cite[Chapter 2, pp. 46 ff.]{davermanvenema1}). It may be convenient to mention that $B$ may well not be \emph{ambient} homeomorphic to $B_0$; that is, there may not exist a homeomorphism of all of $\mathbb{R}^3$ that sends $B$ onto $B_0$.

Let the surface $S$ be the boundary of $B$. Then $S$ is homeomorphic to the original $S_0$ and bounds a $3$--manifold (the manifold $B$ itself). Also, by construction it contains $T$ and, if the feelers are drawn in a locally flat fashion as in Figure \ref{fig:modify_B}, clearly $S$ is locally flat at each point except for, possibly, at those belonging to $T$. Therefore, choosing $T$ to be non rectifiable (for instance, letting $T$ be Antoine's necklace, which is described in Section \ref{sec:cantor}), Theorem \ref{teo:main2} guarantees that $S$ cannot be realized as an attractor.

We will show in Section \ref{sec:proof_main} how to refine the above construction to prove that every closed (orientable) surface can be embedded in $\mathbb{R}^3$ in uncountably many different ways $\{S_i : i \in I\}$ none of which can be realized as an attractor. By ``different'' we mean that for any two different $i \neq j$ there does not exist a homeomorphism of $\mathbb{R}^3$ that sends $S_i$ onto $S_j$. This is true even if one considers only the simplest possible surface, a $2$--sphere:

\begin{corollary} \label{cor:uncountable} There exist uncountably many different $2$--spheres in $\mathbb{R}^3$ that cannot be realized as attractors.
\end{corollary}

We finish this section with some remarks about Theorem \ref{teo:main2}:

\begin{remark} \label{rem:aboutmain} (1) The construction just described illustrates why Theorem \ref{teo:main2} is more convenient than Theorem \ref{teo:main2prima} to establish the nonattracting nature of a given surface. To apply Theorem \ref{teo:main2} we only needed to check that the surface in question was locally flat outside $T$, and that was straightforward. Had we tried to apply Theorem \ref{teo:main2prima}, we would have needed to identify the set $W$ of wild points precisely, which is not easy to do (in fact, $W$ depends on the particular details of how the construction is performed).

(2) The conditions laid out in Theorem \ref{teo:main2} are sufficient, but not necessary, to guarantee that a surface (even a surface with a totally disconnected set of wild points) cannot be an attractor. A suitable example is the horned sphere $S$ of Alexander, which cannot be realized as an attractor (this is proved in \cite{mio6}) but whose set of wild points is a rectifiable Cantor set because it is contained in a straight line by construction. The Alexander sphere is shown in Figure \ref{fig:not_lf} and described carefully in \cite[Chapter 2]{davermanvenema1}. One may think of $S$ as the round sphere $\mathbb{S}^2$ from which infinitely many feelers have been extracted which, as their diameters tend to zero, converge to the Cantor set of wild points at the bottom of the figure. In contrast to the construction ``\`a la Antoine'' described earlier, where the feelers could or could not be entangled, here they must be entangled, for this is what ultimately lends the sphere its wild nature. 

\begin{figure}[h]
\begin{pspicture}(0,0)(7.5,4.2)
\rput[bl](1.25,0){\scalebox{0.75}{\includegraphics{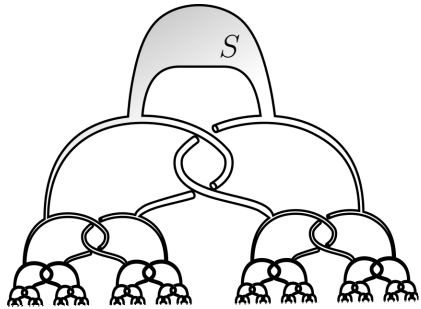}}}
\rput(4.2,3.5){$S$}
\end{pspicture}
\caption{The sphere of Alexander \label{fig:not_lf}}
\end{figure}

This example begs the question of whether it is possible to generalize Theorem \ref{teo:main2} so that it covers both the examples of Antoine and Alexander. Heuristically, such a generalization would somehow capture simultaneously the ``contributions to the wildness'' of $S$ due to both the non rectifiability of $W$ itself and the entanglement of the surface as it approaches $W$.
\end{remark}

The rest of the paper is organized as follows. Proposition \ref{prop:tame_att} is proved in the very brief Section \ref{sec:prop1}, where we also discuss succintly the relation between wildness and tameness. In Section \ref{sec:recall} we recall from \cite{mio6} the definition and properties of certain number $r(K)$ that can be associated to any compact set $K \subseteq \mathbb{R}^3$ and somehow measures its ``crookedness'' as a subset of Euclidean space. This $r(K)$ is finite in the case of attractors, and it is because of this property that it allows us to uncouple the topological and dynamical arguments underlying the proof of Theorem \ref{teo:main2}: dynamics will enter the picture only through the condition that $r$ of an attractor is finite whereas the hard work is on the topological side, in Sections \ref{sec:cantor} and \ref{sec:top}. The first one is devoted to the study of $r(T)$ for totally disconnected compact sets $T$. The main result there is Theorem \ref{teo:tot_disc}, which shows that $r(T)$ is either zero or infinity and relates this to the rectifiability of $T$. In Section \ref{sec:top} we essentially prove the inequality $r(B) \geq r(W)$ (Theorem \ref{teo:top}), where $B \subseteq \mathbb{R}^3$ is a compact $3$--manifold and $W$ is the set of wild points of its boundary $\partial B$. This inequality is at the heart of the proof of the main theorem, which is finally established in Section \ref{sec:proof_main} together with Corollary \ref{cor:uncountable}.

\section{Proof of Proposition \ref{prop:tame_att}} \label{sec:prop1}

In Section \ref{sec:statement} we defined a wild surface $S$ as one having at least a wild point; that is, a point at which the surface is not locally flat. Formally, wildness is not related to local flatness but to local tameness, so we shall devote a few lines to clarify the relation between these concepts.

Let us begin by defining a \emph{polyhedron} $P$ as the geometric realization of a finite simplicial complex in $\mathbb{R}^3$ (although very restrictive, this definition is enough for our purposes). A compact set $K \subseteq \mathbb{R}^3$ is \emph{tame} if there exists an ambient homeomorphism $h : \mathbb{R}^3 \longrightarrow \mathbb{R}^3$ such that $h(K)$ is a polyhedron. And it is \emph{locally tame} at a point $p \in K$ if there exist a (closed) neighbourhood $V$ of $p$ in $\mathbb{R}^3$ and an embedding $e : V \longrightarrow \mathbb{R}^3$ such that $e(V \cap K)$ is a polyhedron. These two definitions are classic; see for instance \cite[p. 145]{moise2}. Evidently a tame set is locally tame at each point. A deep theorem of Bing, proved independently by Moise, states that the converse is also true: a locally tame subset of $\mathbb{R}^3$ is tame \cite[Theorem 4, p. 254]{moise2}.

The above definitions are set up for any compact subset of $\mathbb{R}^3$, but we are interested in the case of closed surfaces $S$. The following holds true for them:

\begin{remark} \label{rem:equiv} For a closed surface $S \subseteq \mathbb{R}^3$, being locally flat at a point $p$ is equivalent to being locally tame at that same point $p$.
\end{remark}
\begin{proof} In this paper we shall only make use of implication ($\Rightarrow$), whose simple proof we give in detail. The converse, although conceptually easy, is rather technical and we only give an intuitive idea.

($\Rightarrow$) By local flatness there exist a neighbourhood $U$ of $p$ in $\mathbb{R}^3$ and a homeomorphism $h$ between $(U,U \cap S)$ and $(\mathbb{R}^3,\mathbb{R}^2 \times \{0\})$. Without loss of generality we can assume that $h(p)$ is the origin. Let $W := [-1,1] \times [-1,1] \times [-1,1]$. Clearly $V := h^{-1}(W)$ is a neighbourhood of $p$ in $\mathbb{R}^3$, and the restriction $e := h|_V$ is an embedding such that $e(V \cap S) = W \cap (\mathbb{R}^2 \times \{0\})$, which is a square (hence a polyhedron). Therefore $S$ is locally tame at $p$.

($\Leftarrow$) Since $S$ is locally tame at $p$, we may think of it (via the embedding $e$ of the definition of local tameness) as a polyhedral surface around $p$. Thus $p$ has a neighbourhood in $S$ that is a polyhedral disk. Such a disk can be flattened out. Therefore $S$ is locally flat around $p$.
\end{proof}

Thus for a surface $S$ one may equivalently define a point $p \in S$ to be wild whenever $S$ is not locally flat or not locally tame at $p$. The former characterization is self contained (it makes no reference to polyhedra), which is why we used it in the previous section. It is also the most convenient definition for later sections. In this section, however, the characterization in terms of polyhedra is more useful.

In order to prove Proposition \ref{prop:tame_att} we need an auxiliary result whose proof we omit:

\begin{lemma} \label{lem:aux} Let $P \subseteq \mathbb{R}^3$ be a polyhedron. Then there exists a flow in $\mathbb{R}^3$ having $P$ as an attractor.
\end{lemma}
\begin{proof} See \cite[Corollary 4, p. 327]{gunthersegal1} or \cite[Proposition 12, p. 6169]{mio5}.
\end{proof}

\begin{proof}[Proof of Proposition \ref{prop:tame_att}] Suppose that $S$ has no wild points, so that by definition it is locally flat at each $p \in S$. By Remark \ref{rem:equiv} this implies that $S$ is locally tame at each $p \in S$ and by the theorem of Bing and Moise mentioned earlier, $S$ is tame. Thus, there exists a homeomorphism $h : \mathbb{R}^3 \longrightarrow \mathbb{R}^3$ such that $h(S)$ is a polyhedron. In turn, by Lemma \ref{lem:aux} there exists a flow having $h(S)$ as an attractor. Let $f$ be the time-one map of this flow, so that $h(S)$ is an attractor for the homeomorphism $f$ of $\mathbb{R}^3$. Then $S$ is an attractor for $h^{-1}fh$.
\end{proof}

\section{The number $r(K)$} \label{sec:recall}

In \cite{mio6} we described how to associate, to each compact subset $K \subseteq \mathbb{R}^3$, a number $r(K)$ that provides a measure of how wildly $K$ sits in $\mathbb{R}^3$. We used it to show that certain arcs, balls and spheres (for instance, the horned sphere of Alexander in Figure \ref{fig:not_lf}.(b)) cannot be attractors. Since this number $r(K)$ will also be the fundamental tool in the proof of Theorem \ref{teo:main2}, we devote this section to review its definition and some of its properties. We refer the reader to \cite{mio6} for more details and proofs.

Let $K \subseteq \mathbb{R}^3$ be an arbitrary compact subset of $\mathbb{R}^3$. It is easy to see that any neighbourhood of $K$ contains a smaller one $N$ which is a compact, polyhedral $3$--manifold. For instance, cover $K$ with the interiors of closed cubes contained in $U$; discard all of them but finitely many using the compactness of $K$, thus obtaining a polyhedral $P$ neighbourhood of $K$, and finally let $N$ be a regular neighbourhood of $P$ (in the sense of piecewise linear topology \cite[Chapter 3, pp. 31 ff. and particularly Proposition 3.10, p. 34]{rourkesanderson1}). We call such an $N$ a \emph{pm}--neighbourhood of $K$ (\emph{pm} standing for polyhedral manifold). For $r$ a nonnegative integer, consider the following property that $K$ may or may not have: \[(P_r) : \text{ $K$ has arbitrarily small \emph{pm}--neighbourhoods $N$ with } {\rm rk}\ H_1(N) \leq r.\] Here $H_1(N)$ denotes the first homology group of $N$ with $\mathbb{Z}_2$ coefficients and ${\rm rk}$ denotes its rank as an Abelian group or, equivalently, its dimension as a vector space over $\mathbb{Z}_2$.

\begin{definition} $r(K)$ is the smallest nonnegative integer $r$ for which $(P_r)$ holds, or $\infty$ if $(P_r)$ does not hold for any $r$.
\end{definition}

Although the definition of $r(K)$ does not make any reference to dynamics, it has the following property which justifies our interest in it:

\begin{itemize}
	\item [(P1)] \emph{Finiteness}. If $K$ is a local attractor for a homeomorphism then $r(K) < \infty$. \label{P1}
\end{itemize}
\begin{proof} See \cite[Theorem 13]{mio6}.
\end{proof}

Essentially, the proof of Theorem \ref{teo:main2} will consist in showing that $r(S) = \infty$ and applying the finiteness property. However, computing $r(K)$ from first principles is in general very difficult and we will need to do it indirectly, making use of several additional properties of $r(K)$ that are, this time, geometric in nature: invariance, semicontinuity, nullity and subadditivity. We state them now.

\begin{itemize}
	\item [(P2)] \emph{Invariance}. If $K$ and $K'$ are ambient homeomorphic, then $r(K) = r(K')$.
\end{itemize}
\begin{proof} See \cite[Theorem 14]{mio6}.
\end{proof}
\begin{itemize}
	\item [(P3)] \emph{Semicontinuity}. Suppose $(K_n)_{n \geq 1}$ is a decreasing sequence (that is, $K_{n+1} \subseteq K_n$ for every $n$) of compact sets and denote $K$ its intersection. Then \[r(K) \leq \liminf_{n \rightarrow \infty}\ r(K_n).\]
\end{itemize}
\begin{proof} Denote $r := \liminf_{n \rightarrow \infty} r(K_n)$. If $r = \infty$ there is nothing to prove, so assume that $r < \infty$. Passing to a subsequence of the $K_n$ we may assume that $r(K_n) = r$ for every $n \in \mathbb{N}$, so that each $K_n$ has property $(P_r)$. Let $U$ be a neighbourhood of $K$. The $K_n$ form a decreasing sequence and $K = \bigcap K_n$, so there exists $n_0$ such that $K_n \subseteq U$ for every $n \geq n_0$. Since $K_{n_0}$ has property $(P_r)$, it has a \emph{pm}--neighbourhood $N \subseteq U$ such that ${\rm rk}\ H_1(N) \leq r$. But $N$ is also a neighbourhood of $K$, so $K$ has property $(P_r)$ too.
\end{proof}
\begin{itemize}
	\item [(P4)] \emph{Nullity}. Suppose $K$ has $\check{H}^2(K) = 0$. Then $r(K) = 0$ if, and only if, $K$ has arbitrarily small \emph{pm}--neighbourhoods $N$ such that each component of $N$ is a $3$--cell. By a $3$--cell we mean a set homeomorphic to the closed unit $3$--ball $\mathbb{B}^3 \subseteq \mathbb{R}^3$.
\end{itemize}
\begin{proof} This is proved in \cite[Theorem 16]{mio6} for $K$ connected (in that case $N$ can be taken to be connected too and it is itself a $3$--cell). The same argument given there covers this slightly more general situation where $K$ is not assumed to be connected.
\end{proof}

Finally, to state the subadditivity property we need to introduce some notation and a definition:

\begin{notation} Consider the following subsets of $\mathbb{R}^3$, shown in Figure \ref{fig:tameaxes}:

\begin{enumerate}
	\item[] $Q_0$ is the parallelepiped $[-1,0] \times [-1,1] \times [-1,1]$,
	\item[] $Q_1$ is the parallelepiped $[0,1] \times [-1,1] \times [-1,1]$,
	\item[] $Q$ is the cube $Q_0 \cup Q_1$,
	\item[] $S$ is the square $Q_0 \cap Q_1 = \{0\} \times [-1,1] \times [-1,1]$,
	\item[] $\dot{S}$ denotes the square $S$ minus its edges.
\end{enumerate}

\begin{figure}[h]
\begin{pspicture}(0,0)(9.5,5)
	\rput[bl](0,0){\scalebox{0.75}{\includegraphics{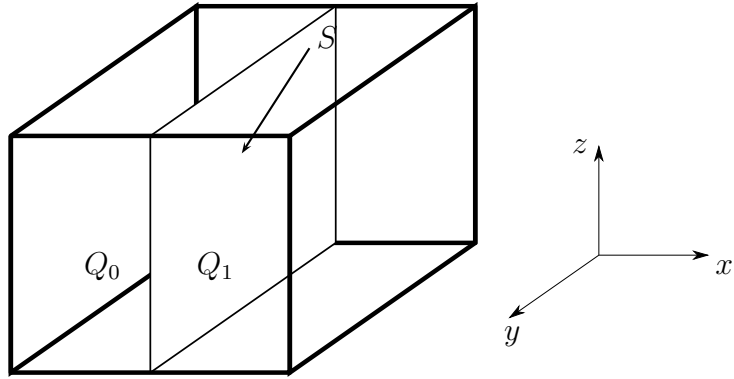}}}
	\psline{->}(4,4.4)(3.1,3) \rput[lb](4.1,4.4){$S$}
	\rput(1.25,1.5){$Q_0$} \rput(2.75,1.5){$Q_1$}
	\rput[bl](9.4,1.4){$x$} \rput[b](7.6,3){$z$} \rput[t](6.7,0.7){$y$}
\end{pspicture}
\caption{The setup for Definition \ref{def:tame} \label{fig:tameaxes}}
\end{figure}

\end{notation}

Suppose $K$ is expressed as the union of two compact sets $K_0$ and $K_1$. We call this a \emph{decomposition} of $K$ and introduce the following definition:

\begin{definition} \label{def:tame} A decomposition $K = K_0 \cup K_1$ is \emph{tame} if
\begin{enumerate}
	\item[({\it i}\/)] $K \cap S = K_0 \cap K_1 \subseteq \dot{S}$,
	\item[({\it ii}\/)] $K_0 \cap Q \subseteq Q_0$ and $K_1 \cap Q \subseteq Q_1$.
\end{enumerate}
\end{definition}

Definition \ref{def:tame} conveys the intuitive idea that $S$ realises \emph{geometrically} the purely set-theoretical decomposition $K = K_0 \cup K_1$. The part of $K$ that lies in $Q$ is structured in two ``halves'', $K \cap Q_0$ and $K \cap Q_1$. The first one sits in $Q_0$ and, because of condition ({\it ii}\/), is comprised exclusively of points from $K_0$. The second one sits in $Q_1$ and is comprised exclusively of points from $K_1$.

\begin{remark} It is convenient to widen Definition \ref{def:tame} slightly and say that a decomposition $K = K_0 \cup K_1$ is tame if there exists an ambient homeomorphism $h : \mathbb{R}^3 \longrightarrow \mathbb{R}^3$ that takes $K$, $K_0$ and $K_1$ onto sets that satisfy ({\it i}\/) and ({\it ii}\/) above. This convenient generalization has no consequences as far as $r(K)$ is concerned since $r$ is invariant under ambient homeomorphisms by (P2).
\end{remark}

A prototypical example of a tame decomposition (in the wider sense just mentioned) of a set $K$ would be as follows:

\begin{example} Take a plane $H \subseteq \mathbb{R}^3$ (having a nonempty intersection with $K$) and, denoting $H_0$ and $H_1$ the two closed halfspaces determined by $H$, let $K_0 := K \cap H_0$ and $K_1 := K \cap H_1$. Then $K = K_0 \cup K_1$ is a tame decomposition.
\end{example}

Now we can state the subadditivity property of $r$:

\begin{itemize}
	\item [(P5)] \emph{Subadditivity}. Let $K = K_0 \cup K_1$ be a tame decomposition of $K$ and assume that $\check{H}^1(K_0 \cap K_1) = 0$. Then the inequality $r(K) \geq r(K_0) + r(K_1)$ holds.
\end{itemize}
\begin{proof} See \cite[Theorem 22]{mio6}.
\end{proof}

If one thinks of $r$ as some sort of Betti number and writes the Mayer-Vietoris sequence for $K = K_0 \cup K_1$, the necessity of requiring that $\check{H}^1(K_0 \cap K_1)$ be zero is clear. The role of the tameness condition is more delicate. Roughly, it guarantees that certain geometrical constructions leading to the inequality $r(K) \geq r(K_0) + r(K_1)$ can be performed. Without this assumption the subadditivity property may be false (see an example in \cite{mio6}).

For our purposes in this paper it is convenient to have a more flexible way of checking whether a decomposition is tame. The following proposition provides this:

\begin{proposition} \label{prop:tame} Let $K = K_0 \cup K_1$ be a decomposition of a compact set $K \subseteq \mathbb{R}^3$. Suppose that there exists an embedding $e : Q \longrightarrow \mathbb{R}^3$ such that:
\begin{enumerate}
	\item $K \cap e(S) = K_0 \cap K_1 \subseteq e(\dot{S})$,
	\item $K_0 \cap e(Q) \subseteq e(Q_0)$ and $K_1 \cap e(Q) \subseteq e(Q_1)$.
\end{enumerate}

Then the decomposition is tame.
\end{proposition}

\label{pg:schon} In the proof we will make use of the Sch\"onflies conjecture, so we devote a few lines to explain it. A $2$--\emph{sphere} $\Sigma \subseteq \mathbb{R}^2$ is simply an embedded copy of the unit two dimensional sphere $\mathbb{S}^2 \subseteq \mathbb{R}^3$. The Sch\"onflies conjecture states that any two $2$--spheres $\Sigma_1$ and $\Sigma_2$ in $\mathbb{R}^3$ are equivalently embedded; that is, there exists an ambient homeomorphism that sends one of them onto the other. As stated, this result is false due to the existence of wild spheres, so some additional hypothesis needs to be added. Brown gave a beautiful version for bicollared spheres. A $2$--sphere $\Sigma \subseteq \mathbb{R}^3$ is \emph{bicollared} if there exists an embedding $b : \Sigma \times [-1,1] \longrightarrow \mathbb{R}^3$ such that $b(p,0) = p$ for every $p \in \Sigma$. Brown proved \cite[Theorem 5, p. 76]{brown1} that the Sch\"onflies conjecture is true for bicollared spheres; that is, if $\Sigma_1, \Sigma_2 \subseteq \mathbb{R}^3$ are two bicollared spheres, there exists an ambient homeomorphism that sends one of them onto the other.

Now let $B_1, B_2 \subseteq \mathbb{R}^3$ be two $3$--cells and $e : B_1 \longrightarrow B_2$ a homeomorphism. Assume that both $\partial B_1$ and $\partial B_2$ are bicollared. Then $e$ admits an extension to a homeomorphism $\hat{e}$ of all $\mathbb{R}^3$. The reason is that the Sch\"onflies conjecture (in Brown's version) allows one to reduce this problem to the case where $B_1 = B_2 = \mathbb{B}^3$, where $\mathbb{B}^3$ denotes the closed unit ball in Euclidean space $\mathbb{R}^3$, and then extending $e$ is a simple matter (do it radially).

\begin{proof} We need to construct a homeomorphism $h$ of $\mathbb{R}^3$ that sends $K_0$ and $K_1$ onto sets satisfying the conditions in Definition \ref{def:tame}. Let $B_1$ and $B_2$ be the $3$--cells $Q$ and $e(Q)$, respectively, and consider the $2$--spheres $\partial B_1$ and $\partial B_2$. Clearly $\partial B_1$ is bicollared, since $B_1$ is a polyhedral cube. As for $\partial B_2$, the following claim shows that we can assume it to be bicollared too:
\medskip

{\it Claim.} Maybe after a suitable improvement of $e$, we may assume that the boundary of $e(Q)$ is bicollared.

{\it Proof of claim.} For any number $0<s<1$ denote by $sQ$ the image of $Q$ under the mapping $(x,y,z) \mapsto (sx,sy,sz)$, and similarly for $sQ_0$, $sQ_1$ and $sS$. It is easy to check that if $s$ is very close to $1$, conditions (1) and (2) in the statement of this proposition are still satisfied when we replace $Q$, $Q_0$, $Q_1$ and $S$ by their scaled down versions $sQ$, $sQ_0$, $sQ_1$ and $sS$. That is,
\begin{enumerate}
	\item[(1$'$)] $K \cap e(sS) = K_0 \cap K_1 \subseteq e(s\dot{S})$,
	\item[(2$'$)] $K_0 \cap e(sQ) \subseteq e(sQ_0)$ and $K_1 \cap e(sQ) \subseteq e(sQ_1)$.
\end{enumerate}

The boundary of $sQ$ is just a polyhedral sphere slightly smaller than the boundary of $Q$, and it is clearly bicollared in the interior of $Q$. Thus $e(sQ)$ is bicollared in $e({\rm int}\ Q)$, which is an open subset of $\mathbb{R}^3$, and so $e(sQ)$ is actually bicollared in $\mathbb{R}^3$. Replacing the embedding $e : Q \longrightarrow \mathbb{R}^3$ by $(x,y,z) \longmapsto e(sx,sy,sz)$ the result follows. $_{\blacksquare}$
\medskip

Now using the Sch\"onflies' theorem we see that $e$ extends to a homeomorphism $\hat{e} : \mathbb{R}^3 \longrightarrow \mathbb{R}^3$ as explained earlier. Let $h := \hat{e}^{-1}$. Applying $h$ throughout to (1) and (2) in the statement of this proposition and cancelling any appearance of $he$ (which is the identity by construction) it follows that the decomposition $h(K) = h(K_0) \cup h(K_1)$ satisfies (1) and (2) of Definition \ref{def:tame}, as was to be proved.
\end{proof}

\section{$r(T)$ when $T$ is a totally disconnected, compact set} \label{sec:cantor}

Throughout this section $T$ will always denote a compact, totally disconnected subset of $\mathbb{R}^3$. Recall that a set is totally disconnected if its connected components are just singletons. As mentioned earlier, we shall say that $T$ is rectifiable if there exists a homeomorphism of $\mathbb{R}^3$ that sends $T$ into a straight line $L$. The following remark follows easily from this definition:

\begin{remark} \label{rem:r0tame} Let $T$ be a compact, totally disconnected set. If $T$ is rectifiable, then it has arbitrarily close neighbourhoods $N$ which are unions of disjoint polyhedral $3$--cells. As a consequence, $\mathbb{R}^3 - T$ is simply connected and also $r(T) = 0$.
\end{remark}
\begin{proof} Up to an ambient homeomorphism we may assume that $T$ lies in a straight line $L$; say the $x$--axis for definiteness. Notice that this does not alter the fundamental group of $\mathbb{R}^3 - T$ or the value of $r(T)$, since both are invariant under ambient homeomorphisms. Let $U$ be a neighbourhood of $T$ in $\mathbb{R}^3$, so that $U \cap L$ is a neighbourhood of $T$ in $L$. Since $T$ is totally disconnected, it has a compact neighbourhood $J$ in $L$ which is a union of finitely many disjoint intervals $J_1, \ldots, J_n$ and satisfies $J \subseteq U \cap L$. This $J$ can be thickened to obtain a \emph{pm}--neighbourhood $N$ of $T$ in $\mathbb{R}^3$ simply choosing $\varepsilon > 0$ so small that $N := J \times [-\varepsilon,\varepsilon] \times [-\varepsilon,\varepsilon]$ is still contained in $U$. Clearly each component of $N$ is, by construction, a polyhedral $3$--cell. Letting $U$ vary we see that $T$ has arbitrarily close neighbourhoods $N$ with this property. In particular, since these satisfy $H_1(N) = 0$, it follows that $r(T) = 0$. Also, the complement in $\mathbb{R}^3$ of any of these neighbourhoods $N$ is simply connected, and so the same is true of the complement of $T$.
\end{proof}

The main theorem in this section strengthens the previous remark as follows:

\begin{theorem} \label{teo:tot_disc} Let $T \subseteq \mathbb{R}^3$ be a compact, totally disconnected set. Then the following alternative holds:
\begin{itemize}
	\item[({\it i}\/)] if $T$ is rectifiable, then $r(T) = 0$,
	\item[({\it ii}\/)] if $T$ is not rectifiable, then $r(T) = \infty$.
\end{itemize}
Thus $r(T)$ cannot take any intermediate values between $0$ and $\infty$.
\end{theorem}

 Before proving the theorem it may be instructive to examine the particular example of Antoine's necklace $C$, already mentioned in Section \ref{sec:statement}, because we can easily check that it is not rectifiable and also show, exploiting its self similarity properties, that $r(C) = \infty$. A detailed exposition can be found in the original paper by Antoine \cite[\S 78, p. 311 ff.]{antoine1} or more modern references such as \cite[pp. 42 ff.]{davermanvenema1} or \cite[\S 18, pp. 127 ff.]{moise2}.

Antoine's necklace is obtained as the intersection of a decreasing sequence of compact manifolds $N_k$ each of which is a ``necklace'' comprised of several linked tori $T_{k,j}$. The first one, $N_0$, consists of a single unknotted solid torus $T_0 \subseteq \mathbb{R}^3$. The next one, $N_1$, is the union of $n$ of solid tori $T_{1,1},T_{1,2},\ldots,T_{1,n}$ contained in the interior of $T_0$ and linked as shown in Figure \ref{fig:antoine1} for $n = 5$ (the drawing shows only one $T_{1,j}$ in full and just the cores of remaining tori). Place a similar (but scaled down) arrangement of $n$ linked solid tori $T_{2,j}$ inside each of the $T_{1,j}$. Then $N_2$ is the union of all these second generation tori, of which there are $n^2$ in total. Repeating the construction inductively yields the decreasing sequence of sets $N_k$ and their intersection $C = \bigcap_k N_k$ is Antoine's necklace. Notice that the diameter of the connected components of the $N_k$ approaches zero as $k \longrightarrow +\infty$. This implies that $C$ is totally disconnected (in fact, it is a Cantor set).

\begin{figure}[h]
\begin{pspicture}(0,0)(10,4)
\rput[bl](0,0){\scalebox{0.75}{\includegraphics{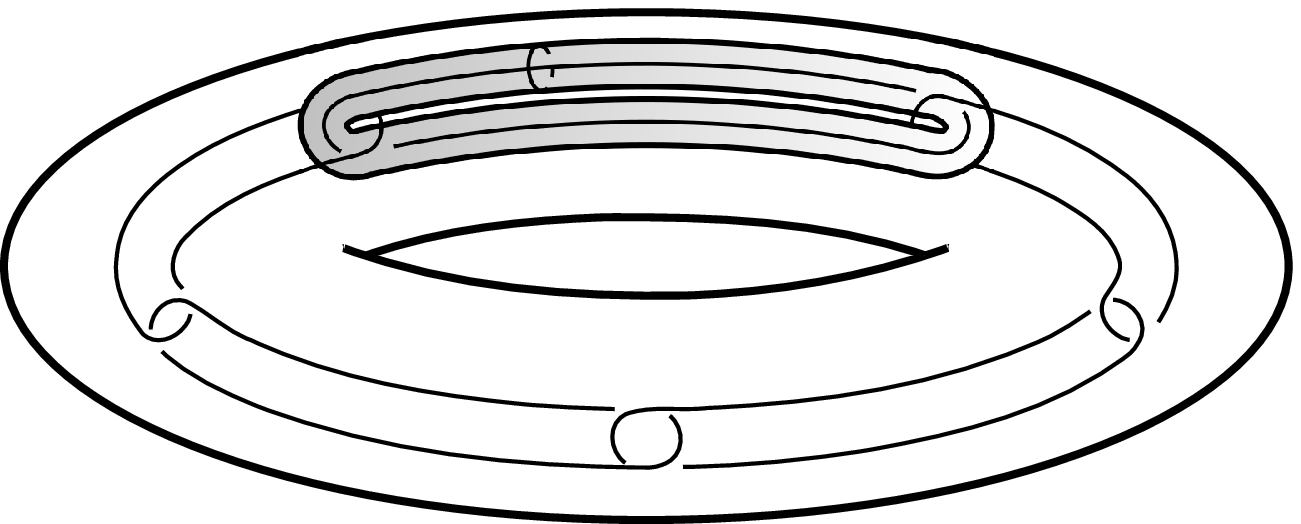}}}
\rput(9.5,0.6){$T_0$}
\rput[bl](5.2,2.4){$T_{1,j}$}
\end{pspicture}
\caption{\label{fig:antoine1}}
\end{figure}

The crucial property of $C$ is that its complement in $\mathbb{R}^3$ is not simply connected. A formal proof can be found in the book by Moise \cite[Theorem 4, p. 131]{moise2}, but it is easy to convince oneself intuitively as follows. Consider a meridian $\mu$ of $T_0$ (contained in the surface of $T_0$). If we tried to contract it to a point we would surely run into some of the $T_{1,j}$ because of how they are linked to form a closed chain. On a finer scale, and for the same reason, we would also run into some of the $T_{2,j}$, and so on. Thus $\mu$ is not contractible in the complement of any of the $N_k$, and so it cannot be contractible in the complement of $C$ either.

It follows from the previous paragraph and Remark \ref{rem:r0tame} that $C$ is not rectifiable. Let us check also that $r(C) = \infty$, as it should be according to Theorem \ref{teo:tot_disc}. By the very nature of the construction of $C$ each $T_{1,j}$ contains a Cantor set $C_j$ that is ambient homeomorphic to all of $C$; namely $C_j = C \cap T_{1,j}$. Hence $r(C) = r(C_j)$ for every $1 \leq j \leq n$. Moreover $C$ is the disjoint union of the closed sets $C_1, C_2, \ldots, C_n$ so $r(C) = \sum_{j=1}^n r(C_j) = n r(C)$. Thus either $r(C) = 0$ or $r(C) = \infty$. Suppose $r(C) = 0$ were true. Then, by the nullity property (P4), $C$ would have arbitrarily small \emph{pm}--neighbourhoods $N$ that are unions of $3$--balls. The complement in $\mathbb{R}^3$ of any such neighbourhood is clearly simply connected, and therefore the same would be true of $C$; that is, $\mathbb{R}^3 - C$ would be simply connected. Since this is not the case, we conclude that $r(C) \geq 1$ and so $r(C) = \infty$.

In proving Theorem \ref{teo:tot_disc} we shall make use of two auxiliary lemmas which we state below as Lemmas \ref{lem:r0tame} and \ref{lem:dico}. Both of them can be found in a paper by Bing \cite{bing4}. Although he proves them for Cantor sets, the argument applies to any compact, totally disconnected set $T$. Also, we warn the reader that Bing uses the more standard terminology ``tame'' for what we call ``rectifiable'' (we find the latter more expressive and, in this paper, less prone to confusion with other usages of ``tame'').

\begin{lemma} \label{lem:r0tame} Let $T \subseteq \mathbb{R}^3$ be a compact, totally disconnected set. Then $T$ is rectifiable if, and only if, it has arbitrarily small \emph{pm}--neighbourhoods $N$ such that every component of $N$ is a $3$--cell.
\end{lemma}
\begin{proof} See the steps outlined in the proof of \cite[Theorem 1.1, p. 435 ff.]{bing4}.
\end{proof}

Let us say that $T$ is \emph{locally rectifiable} at a point $p \in T$ if there exist a neighbourhood $U$ of $p$ in $\mathbb{R}^3$ and a homeomorphism of $\mathbb{R}^3$ that sends $U \cap T$ into a straight line.

\begin{lemma} \label{lem:dico} Let $T \subseteq \mathbb{R}^3$ be a compact, totally disconnected set. If $T$ is locally rectifiable at every point except for, possibly, a finite set of points $F \subseteq T$, then $T$ is rectifiable.
\end{lemma}
\begin{proof} See \cite[Theorem 4.1, p. 440]{bing4} and the remark following it.
\end{proof}

Now we can prove the main result of this section.

\begin{proof}[Proof of Theorem \ref{teo:tot_disc}] Part ({\it i}\/) is already contained in Remark \ref{rem:r0tame}; we only need to prove part ({\it ii}\/), and we are going to establish its contrapositive. Therefore, assume that $r(T) = r < \infty$. Choose a \emph{pm}--neighbourhood basis $\{N_k\}$ of $T$ such that the following properties are satisfied:
\begin{enumerate}
	\item[(N1)] $N_{k+1} \subseteq {\rm int}\ N_k$ for each $k$,
	\item[(N2)] each component of every $N_k$ meets $T$,
	\item[(N3)] ${\rm rk}\ H_1(N_k) = r$ for every $k$,
	\item[(N4)] for every \emph{pm}--neighbourhood $N$ of $T$ contained in $N_1$, ${\rm rk}\ H_1(N) \geq r$.
\end{enumerate}

That such a \emph{pm}--neighbourhood basis $\{N_k\}$ exists follows easily from the definition of $r(T)$. It suffices to show that $T$ is locally rectifiable at every point $p \in T$ except for, possibly, a finite number of them; then Lemma \ref{lem:dico} entails that $T$ is actually rectifiable, proving the theorem. We begin with two auxiliary claims:
\medskip

{\it Claim 1.} Let $\ell \geq k$. Denote $C_1, \ldots, C_n$ the components of $N_k$, and for each $1 \leq i \leq n$ denote $D_{ij}$ (for $1 \leq j \leq m_i$) the components of $N_{\ell}$ that lie in $C_i$. Then, for every $1 \leq i \leq n$, \[{\rm rk}\ H_1(C_i) = \sum_{j=1}^{m_i} {\rm rk}\ H_1(D_{ij}).\]
\smallskip

{\it Proof.} First we prove the inequality ${\rm rk}\ H_1(C_i) \leq \sum_j {\rm rk}\ H_1(D_{ij})$. Fix some $i_0$. Consider the \emph{pm}--neighbourhood $N$ of $T$ obtained from $N_k$ by deleting $C_{i_0}$ and replacing it with the $D_{i_0 j}$. That is, let \[N := C_1 \cup \ldots \cup C_{i_0 - 1} \cup \underbrace{D_{i_0 1} \cup \ldots \cup D_{i_0 m_{i_0}}}_{\text{in place of $C_{i_0}$}} \cup C_{i_0 + 1} \cup \ldots \cup C_n.\] Then clearly \[{\rm rk}\ H_1(N) = {\rm rk}\ H_1(N_k) - {\rm rk}\ H_1(C_{i_0}) + \sum_{j=1}^{m_{i_0}} {\rm rk}\ H_1(D_{i_0 j}).\] By (N3) we have ${\rm rk}\ H_1(N_k) = r$, and by (N4) we also have ${\rm rk}\ H_1(N) \geq r$. Replacing these above, \[r \leq {\rm rk}\ H_1(N) = r - {\rm rk}\ H_1(C_{i_0}) + \sum_{j=1}^{m_{i_0}} {\rm rk}\ H_1(D_{i_0 j}),\] which implies \begin{equation} \label{eq:1} {\rm rk}\ H_1(C_{i_0}) \leq \sum_{j=1}^{m_{i_0}} {\rm rk}\ H_1(D_{i_0 j}).\end{equation}

Now observe that $N_k$ is the disjoint union of the $C_i$ and $N_{\ell}$ is the disjoint union of the $D_{ij}$, so \[{\rm rk}\ H_1(N_k) = \sum_{i=1}^n {\rm rk}\ H_1(C_i) \quad \text{and} \quad {\rm rk}\ H_1(N_{\ell}) = \sum_{i=1}^n \sum_{j=1}^{m_i} {\rm rk}\ H_1(D_{ij}).\] Since ${\rm rk}\ H_1(N_k) = {\rm rk}\ H_1(N_{\ell}) = r$ by (N3), it follows that \[\sum_{i=1}^n {\rm rk}\ H_1(C_i) = \sum_{i=1}^n \sum_{j=1}^{m_i} {\rm rk}\ H_1(D_{ij}),\] or \[\sum_{i=1}^n \left( {\rm rk}\ H_1(C_i) - \sum_{j=1}^{m_i} {\rm rk}\ H_1(D_{ij}) \right) = 0.\] The inequality \eqref{eq:1} established above implies that each of the terms in this sum is nonpositive, so it follows that they must all be zero. This proves the claim. $_{\blacksquare}$
\medskip

{\it Claim 2.} Let $C$ be a component of some $N_{k_0}$ and assume that ${\rm rk}\ H_1(C) = 0$. Then $C \cap T$, which is a compact and totally disconnected set, is rectifiable.
\smallskip

{\it Proof.} Denote $N'_k := N_k \cap C$ for every $k \geq k_0$. Clearly $\{N'_k\}$ is a neighbourhood basis of $C \cap T$. Moreover, since each $N'_k$ is nothing but the union of those components of $N_k$ that lie in $C$, they are all polyhedral manifolds and so $\{N'_k\}$ is actually a \emph{pm}--neighbourhood basis of $C \cap T$. Also, from Claim 1 and the hypothesis ${\rm rk}\ H_1(C) = 0$ it follows that ${\rm rk}\ H_1(N'_k) = 0$ for every $k$. Therefore $r(C \cap T) = 0$ and, by the nullity property (P4), $C \cap T$ has arbitrarily small \emph{pm}--neighbourhoods $N$ such that every component of $N$ is a $3$--ball. Thus by Lemma \ref{lem:r0tame}, $C \cap T$ is rectifiable. $_{\blacksquare}$
\medskip

Now we can complete the proof of the theorem. For each point $p$ in $T$ let $C(N_k;p)$ be the connected component of $N_k$ that contains $p$. By Claim 1 above the sequence ${\rm rk}\ H_1(C(N_k;p))$ decreases as $k$ increases, so we can define \[s(p) := \lim_{k \rightarrow \infty} {\rm rk}\ H_1(C(N_k;p)).\]

Observe that at most $r$ different points $p \in T$ have $s(p) \neq 0$. Indeed, suppose that at least $p_1, \ldots, p_{r+1} \in T$ had $s(p_j) \geq 1$. Choose $k$ big enough so that all the $p_j$ lie in different components $C(N_k;p_j)$ of $N_k$. Evidently \[{\rm rk}\ H_1(N_k) \geq \sum_{j=1}^{r+1} {\rm rk}\ H_1(C(N_k;p_j)).\] However ${\rm rk}\ H_1(C(N_k;p_j)) \geq s(p_j) \geq 1$ for each $1 \leq j \leq r$ (by definition of $s$), so we conclude that ${\rm rk}\ H_1(N_k) \geq r+1$, contradicting condition (N4) in the choice of $\{N_k\}$.

Now let $p \in T$ be a point for which $s(p) = 0$. We have just seen that this happens for all but finitely points in $T$. Since $s(p)$ is the limit of a sequence of integers, for some $k_0$ we have ${\rm rk}\ H_1(C(N_k;p)) = s(p) = 0$ for $k \geq k_0$. By Claim 2 we see that $C(N_{k_0};p) \cap T$, which is a neighbourhood of $p$ in $T$, is rectifiable. Otherwise stated, $T$ is locally rectifiable at $p$. Thus $T$ is locally rectifiable at every point except for, possibly, a finite number of them, as was to be shown.
\end{proof}

\section{A topological theorem} \label{sec:top}

This section is devoted to the proof of the following result, whose notation and hypotheses will be tacitly assumed to hold throughout:

\begin{theorem} \label{teo:top} Let $B \subseteq \mathbb{R}^3$ be a compact $3$--manifold with a connected boundary $\partial B$. Notice that $\partial B$ is a closed surface. Suppose that $\partial B$ contains a compact, totally disconnected set $T$ such that $\partial B$ is locally flat at each $p \not \in T$. Then $r(B) \geq r(T)$.
\end{theorem}

The proof involves a somewhat delicate geometric construction. In an attempt to convey the intuitive ideas as clearly as possible we have divided the explanation into three stages which we first outline and later on expand in detail.
\medskip

\fbox{Stage 1} Suppose $E \subseteq \partial B$ is a closed disk whose boundary $\partial E$ does not meet $T$, so that $\partial B$ is locally tame at each $p \in \partial E$. We shall show how to:
\begin{itemize}
	\item[(1)] Push $E$ slightly into $B$ while keeping its boundary fixed, obtaining a disk $\hat{E}$ \emph{properly embedded} in $B$. This means that the interior of $E$ is contained in the interior of $B$ and the boundary of $E$ is contained in the boundary of $B$.
	\item[(2)] The disk $\hat{E}$ separates $B$ into two connected components $V_0$ and $V_1$ such that $B = \bar{V}_0 \cup \bar{V}_1$ and $\bar{V}_0 \cap \bar{V}_1 = \hat{E}$. We choose the labeling in such a way that $V_1$ is the component bounded by $E$ and $\hat{E}$.
	\item[(3)] The decomposition $B = \bar{V}_0 \cup \bar{V}_1$ is tame and the condition $\check{H}^1(\bar{V}_0 \cap \bar{V}_1) = 0$ needed to apply the subadditivity property (P5) of $r$ is met. Hence $r(B) \geq r(\bar{V}_0) + r(\bar{V}_1)$.
\end{itemize}

Part (1) relies heavily on the fact that $\partial B$ is the boundary of $B$ and this provides a natural way of pushing $E$ into $B$. It is the ultimate reason why in Theorem \ref{teo:main2} we need to require that $S$ bounds a $3$--manifold on one side (that would be $B$, and then $S = \partial B$). The assumption that $\partial E$ is disjoint from $T$ guarantees that $\partial B$ is locally flat at each $p \in \partial E$ and is used crucially in (3) to prove that the decomposition $B = \bar{V}_0 \cup \bar{V}_1$ is tame.
\medskip

\fbox{Stage 2} This is essentially the same as before, but now instead of a single disk $E \subseteq \partial B$ we consider a finite number of disjoint closed disks $E_1,\ldots,E_n \subseteq \partial B$ such that none of their boundaries $\partial E_j$ meets $T$. In Figure \ref{fig:split_disk}.(a) the set $T$ is schematically shown as a collection of thick points in $\partial B$ and the disks $E_j$ as three disjoint intervals. In our actual application $T$ will be contained in the union of the disks $E_j$, and that is how it is shown in Figure \ref{fig:split_disk}. As before, one can:
\begin{itemize}
	\item[(1)] Push the disks $E_j$ slightly into $B$ while keeping their boundaries fixed. This yields a family of disjoint closed disks $\hat{E}_j$ properly embedded in $B$. The $\hat{E}_j$ are shown in Figure \ref{fig:split_disk}.(b).
	\item[(2)] The (union of the) family of disks $\hat{E}_j$ separates $B$ into $(n+1)$ connected components: a ``big'' one $V_0$ and ``smaller'' ones $V_1,V_2,\ldots,V_n$. The latter are the ones bounded by each $E_j$ and its corresponding $\hat{E}_j$ (see Figure \ref{fig:split_disk}.(c)). As before, $B$ is the union of the closures $\bar{V}_j$ of the $V_j$. Also, $\bar{V}_1, \ldots, \bar{V}_n$ are pairwise disjoint and each of them intersects $\bar{V}_0$ in its corresponding disk $\hat{E}_j$.
	\item[(3)] An argument involving the subadditivity property of $r$ proves that the inequality \[r(B) \geq \sum_{j=0}^n r(\bar{V}_j)\] holds.
\end{itemize}

\begin{figure}[h]
\subfigure[The disks $E_j$]{
\begin{pspicture}(0,0)(8,5)
	\rput[bl](0.25,0.25){\scalebox{0.75}{\includegraphics{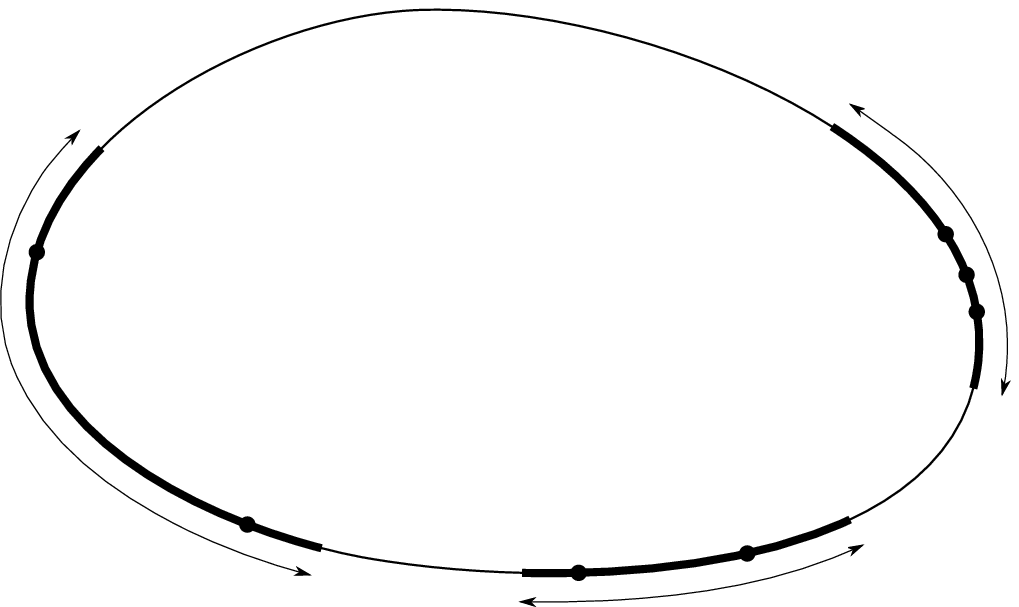}}}
	\rput(1.4,4.6){$\partial B$} \rput(3.8,3.8){$B$}
	\rput(0.6,1.2){$E_1$} \rput(6,0.1){$E_2$} \rput(7.55,3.8){$E_3$}
\end{pspicture}}
\subfigure[The disks $\hat{E}_j$]{
\begin{pspicture}(0,0)(8,5)
	\rput[bl](0.5,0.5){\scalebox{0.75}{\includegraphics{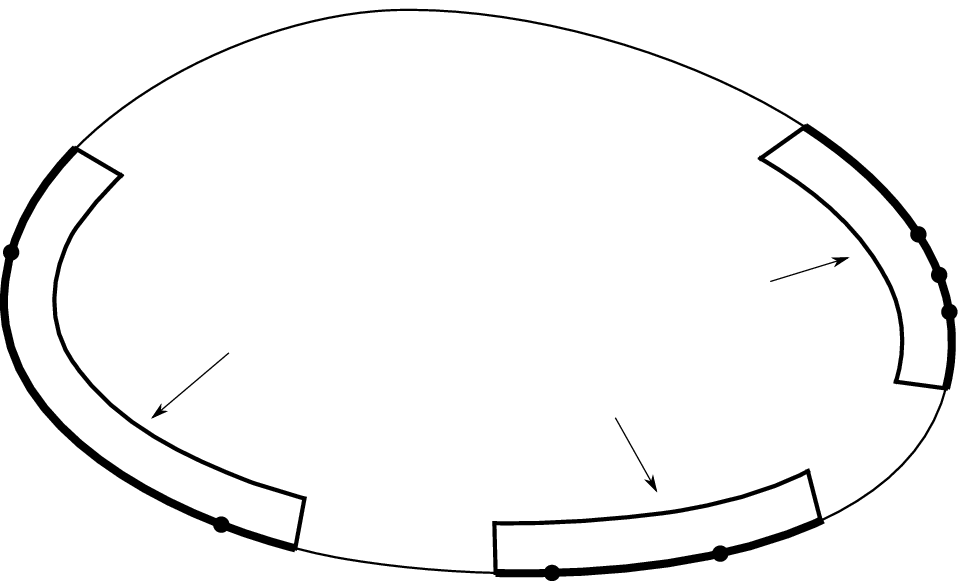}}}
	\rput[bl](2.3,2.3){$\hat{E}_1$} \rput[br](5.1,1.8){$\hat{E}_2$} \rput[tr](6.45,2.8){$\hat{E}_3$}
\end{pspicture}}
\subfigure[The components $V_j$]{
\begin{pspicture}(0,0)(8,6)
	\rput[bl](0.5,0.5){\scalebox{0.75}{\includegraphics{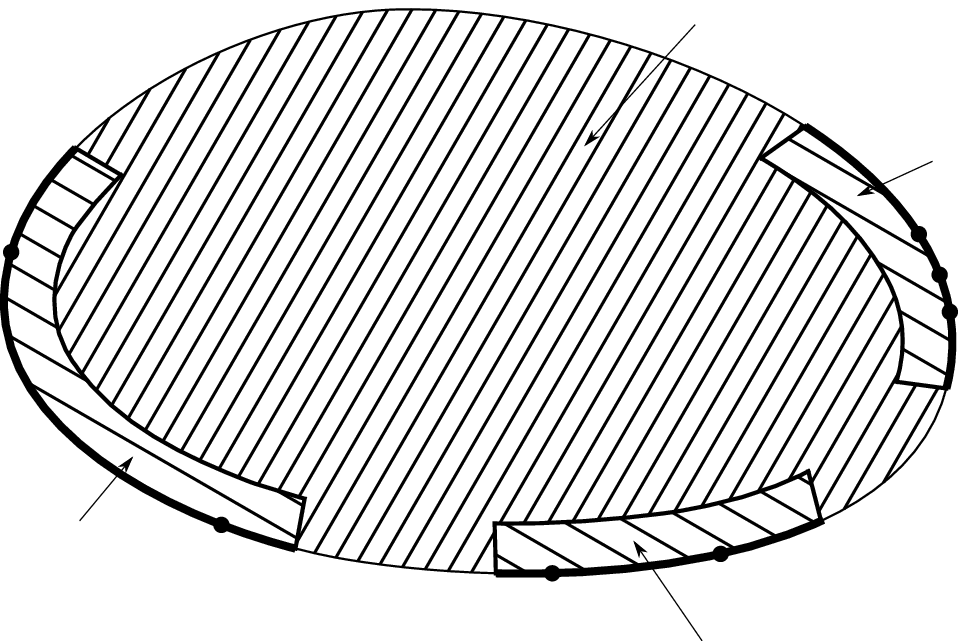}}}
	\rput[tr](1.1,1.4){$V_1$} \rput[tl](5.9,0.5){$V_2$} \rput[bl](5.9,5.2){$V_0$} \rput[bl](7.7,4.2){$V_3$}
\end{pspicture}}
\subfigure[Two stages ($k=1,2$) of the construction]{
\begin{pspicture}(-0.2,0)(8,6)
	\rput[bl](0.5,0.5){\scalebox{0.75}{\includegraphics{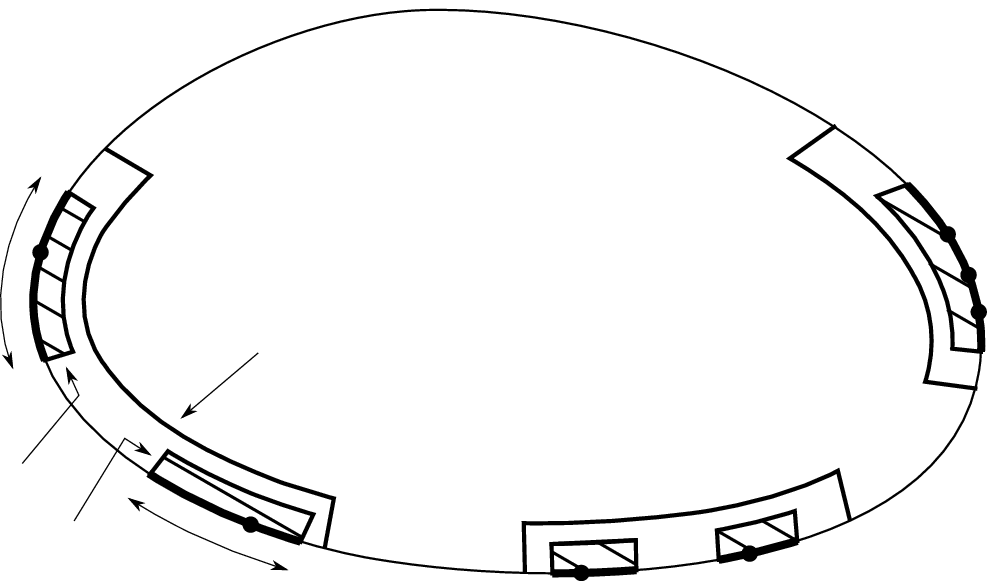}}}
	\rput[bl](2.5,2.2){$\hat{E}^{(1)}_1$}
	\rput[bl](-0.2,2.7){$E_1^{(2)}$} \rput[bl](1.6,0.2){$E_2^{(2)}$}
	\rput[bl](0.1,0.8){$\hat{E}_1^{(2)}$} \rput[bl](0.7,0.4){$\hat{E}_2^{(2)}$}
\end{pspicture}}
\caption{\label{fig:split_disk}}
\end{figure}
\medskip

\fbox{Stage 3} Finally, instead of performing the above construction just once, we perform it for each $k = 1,2,\ldots$ as follows. At each step $k$, we cover $T$ with (the interiors of) a finite family of disks $E_1^{(k)},E_2^{(k)},\ldots,E_{n_k}^{(k)}$ which now depend on $k$. These disks should be chosen in such a way that each $E_j^{(k+1)}$ is contained in an $E_{j'}^{(k)}$ and their diameters tend to zero as $k$ increases. This is always possible because $T$ is a totally disconnected set. Then, paralleling the previous stages:
\begin{itemize}
	\item[(1)] Push each disk $E_j^{(k)}$ slightly into $B$ while keeping its boundary fixed, taking the precaution to push it only to ``depth'' $1/k$. See Figure \ref{fig:split_disk}.(d).
	\item[(2)] For each $k$ the family of disks $E_1^{(k)},E_2^{(k)},\ldots,E_{n_k}^{(k)}$ separates $B$ into $(n_k+1)$ connected components, of which we denote $V_0^{(k)}$ the ``big'' one and $V_1^{(k)},V_2^{(k)},\ldots,V_{n_k}^{(k)}$ the remaining ones.
	\item[(3)] Consider the closures $\bar{V}_j$ of the $V_j$. As before, the inequality \[r(B) \geq \sum_{j=0}^{n_k} r(\bar{V}_j^{(k)})\] holds for each $k$. For the sake of brevity, set \[B_0^{(k)} := \bar{V}^{(k)}_0 \quad \text{ and } \quad B_1^{(k)} := \bar{V}^{(k)}_1 \cup \ldots \cup \bar{V}_{n_k}^{(k)}.\] Since $\bar{V}_1^{(k)}, \ldots, \bar{V}_{n_k}^{(k)}$ are mutually disjoint, clearly $r(B_1^{(k)})$ coincides with the sum $\sum_{j=1}^{n_k} r(\bar{V}_j^{(k)})$ and therefore the above inequality can be written more compactly as \[r(B) \geq r(B_0^{(k)}) + r(B_1^{(k)}).\]
	\item[(4)] From the above inequality we have $r(B) \geq r(B_1^{(k)})$ for every $k$. Since each $E_j^{(k+1)}$ is contained in some $E_{j'}^{(k)}$ and we have pushed $E_j^{(k+1)}$ into $B$ to depth $1/(k+1)$, which is less than we did with $E_{j'}^{(k)}$, it follows that $V_j^{(k+1)}$ is contained in $V_{j'}^{(k)}$ as suggested by Figure \ref{fig:split_disk}.(d). As a consequence, $B_1^{(k+1)} \subseteq B_1^{(k)}$. Moreover, since the diameters of the $E_j^{(k)}$ and the depths $1/k$ of the $\hat{E}_j^{(k)}$ converge to zero, $T = \bigcap_{k =1}^{\infty} B_1^{(k)}$. Then, appealing to the semicontinuity property (P3) of $r$ and using $r(B) \geq r(B_1^{(k)})$, we get \[r(B) \geq \liminf_{k \rightarrow \infty} r(B_1^{(k)}) \geq r(T)\] which proves Theorem \ref{teo:top}.
\end{itemize}
\smallskip

In the remaining of this section we shall fill in the details of the proof of Theorem \ref{teo:top} along the lines just described. We shall begin with a ``Stage 0'' in order to introduce some notation and constructions that will be useful later on. Also, we shall use without explanation some facts that follow from the $2$--dimensional Sch\"onflies conjecture. It was mentioned in page \pageref{pg:schon} that in three dimensions the Sch\"onflies conjecture is false in general. However, in two dimensions it is true: given any two $1$--spheres in $\mathbb{R}^2$ (that is, any two simple closed curves in the plane) there is an ambient homeomorphism that takes one onto the other (see \cite[Chapter 10, pp. 71 ff.]{moise2}). As a consequence, given any two closed disks in the plane there is always an ambient homeomorphism that sends one onto the other. In this form, the Sch\"onflies theorem generalizes to any connected surface: given any two closed disks $E$ and $E'$ in a closed connected surface $S$, there is a homeomorphism of the surface that sends one onto the other. In particular, let $E' := \varphi^{-1}(\mathbb{D}^2)$ where $\varphi : U \subseteq S \longrightarrow \mathbb{R}^2$ is some chart of $S$ and $\mathbb{D}^2$ is the closed unit disk in $\mathbb{R}^2$. By the Sch\"onflies theorem, there is a homeomorphism $g$ of $S$ that sends $E$ onto $E'$. The composition $\varphi \circ g$ is defined on the neighbourhood $g^{-1}(U)$ of $E$ and sends $E$ onto $\mathbb{D}^2 \subseteq \mathbb{R}^2$, allowing us to pull back any construction performed with $\mathbb{D}^2$ to a construction performed with $E$. For instance: (i) one can clearly find arbitrarily thin annuli along $\partial \mathbb{D}^2$; pulling back via $\varphi \circ g$, the same is true of $\partial E$; (ii) one can clearly find arbitrarily small neighbourhoods of $\mathbb{D}^2$ that are homeomorphic to $\mathbb{R}^2$ (for instance, open disks with radius slightly bigger than $1$ but arbitrarily close to it), so the same is true of $E$.

\subsection{Stage 0} \label{sec:stage0} Consider the plane $\{z = 0\}$ in $\mathbb{R}^3$ and a closed disk $D \subseteq \{z = 0\}$. Denote $P := \mathbb{R}^2 \times [-2,2]$ and $P^+ := \mathbb{R}^2 \times [0,2]$. Set \begin{equation} \label{eq:def_D} \hat{D} := \partial D \times [0,1] \cup D \times 1.\end{equation} (To avoid ambiguities in interpreting this and subsequent expressions we adhere to the convention that $\partial$ has precedence over $\times$, which in turn has precedence over $\cup$ and $-$). See Figure \ref{fig:Dhat}. Clearly $\hat{D}$ is a $2$--disk whose boundary coincides with $\partial D$ and whose interior is entirely contained in the upper halfspace $z > 0$; that is, $\hat{D}$ is properly embedded in $P^+$. We say that $\hat{D}$ has been obtained by \emph{pushing $D$ into $P^+$ to depth one}. Notice that $\hat{D}$ separates $P^+$ into two connected components \begin{equation} \label{eq:def_U} U_0 := \mathbb{R}^2 \times (1,2] \cup (\mathbb{R}^2 - D) \times [0,2] \quad \text{and} \quad U_1 := \dot{D} \times [0,1)\end{equation} whose intersection is precisely the disk $\hat{D}$. Here (and in the sequel) a dot over a manifold will denote the interior of the manifold (that is, the manifold minus its boundary). That $\hat{D}$ separates $P^+$ as described should be clear intuitively and can be confirmed by Figure \ref{fig:Dhat}.

\begin{figure}[h]
\begin{pspicture}(-1.5,-0.5)(11.5,4)
	\rput[bl](0,0){\scalebox{0.75}{\includegraphics{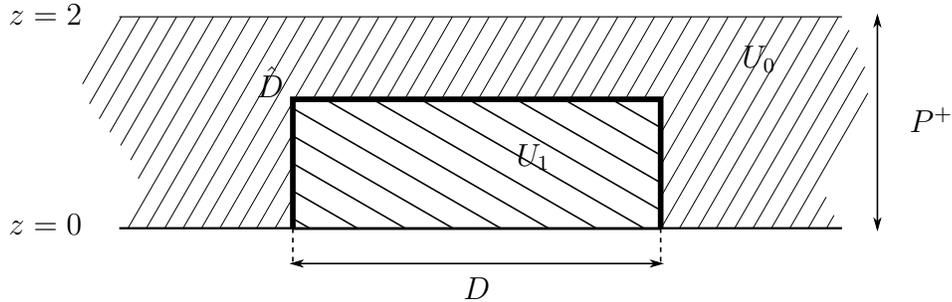}}}
	\rput[r](0,0.6){$z = 0$}
	\rput[r](0,3.4){$z = 2$}
	\rput[l](11,2){$P^+$}
	\rput(9,2.8){$U_0$} \rput(6,1.5){$U_1$}
	\rput(5.25,-0.25){$D$} \rput(2.5,2.5){$\hat{D}$}
\end{pspicture}
\caption{The construction of $\hat{D}$ \label{fig:Dhat}}
\end{figure}

Although it does not really make sense to say that the decomposition $P^+ = U_0 \cup U_1$ is tame because $P^+$ is not compact, the way in which the sets $U_0$ and $U_1$ intersect along $\hat{D}$ is certainly characteristic of tame decompositions. This is the motivation behind the following proposition, where we make use again of the sets $Q$, $S$, $Q_0$ and $Q_1$ introduced when discussing the subadditivity property, just before Figure \ref{fig:tameaxes}.

\begin{proposition} \label{prop:step0} There exists an embedding $e : Q \longrightarrow \mathbb{R}^3$ such that:
\begin{enumerate}
	\item[(1)] $\hat{D} \subseteq e(\dot{S})$,
	\item[(2)] $U_0 \cap e(Q) \subseteq e(Q_0)$ and $U_1 \cap e(Q) \subseteq e(Q_1)$.
\end{enumerate}
\end{proposition}

Figure \ref{fig:step0} below illustrates the content of Proposition \ref{prop:step0} in a two dimensional (very schematic) picture. We have denoted $S' = e(S)$ and $Q' = e(Q)$. $S'$ separates $Q'$ into two connected components whose closures are precisely $e(Q_0)$ and $e(Q_1)$.

\begin{figure}[h]
\begin{pspicture}(-1.5,-0.5)(11.5,6)
	\rput[bl](0,0){\scalebox{0.75}{\includegraphics{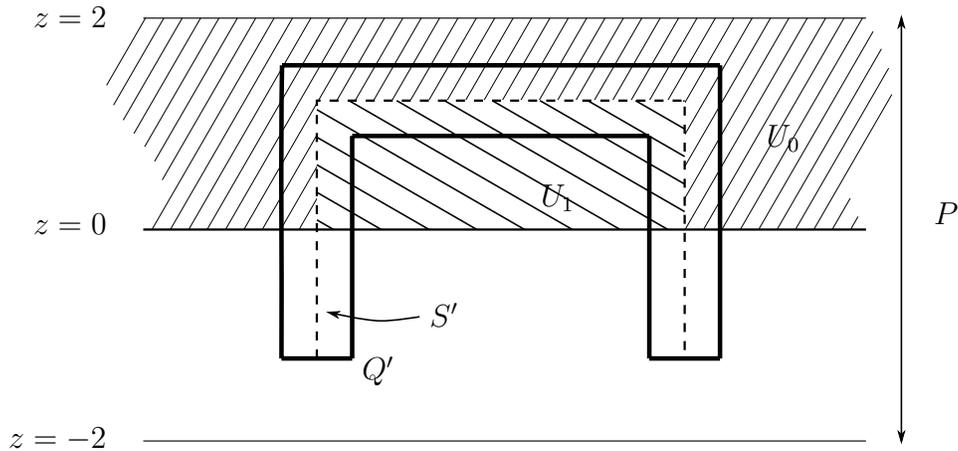}}}
	\rput[r](0,0.1){$z = -2$}
	\rput[r](0,2.95){$z = 0$}
	\rput[r](0,5.7){$z = 2$}
	\rput[l](11,3.1){$P$}
	\rput(9,4.1){$U_0$} \rput(6,3.3){$U_1$}
	\rput[l](3.4,1){$Q'$} \rput[l](4.3,1.75){$S'$}
\end{pspicture}
\caption{Scheme for Proposition \ref{prop:step0} \label{fig:step0}}
\end{figure}

\begin{proof} Let $A$ be a closed annulus along the boundary of $D$. Consider the sets \[S' := \partial D \times [-1,1] \cup D \times 1 \quad \text{ and } \quad Q' := A \times [-1,\nicefrac{3}{2}] \cup D \times [\nicefrac{1}{2},\nicefrac{3}{2}].\]

$Q'$ resembles a ``thick bucket'' turned upside down. Its intersection with the plane $\{z=0\}$ is precisely the annulus $A$. $S'$ is a $2$--disk properly embedded in $Q'$. The intersection of $S'$ with the plane $\{z=0\}$ is precisely the boundary of the disk $D$. The disk $\hat{D}$ is the part of $S'$ that lies above the $\{z=0\}$ plane; that is, the intersection $S' \cap \{z \geq 0\}$. In particular $\hat{D}$ is contained in the interior of $S'$.
\smallskip

{\it Simple case.} First consider the particular case when $D$ is a square and $A$ is the annulus comprised between two homothetic copies of $D$, one slightly bigger and the other slightly smaller than $D$ itself. Figure \ref{fig:prime} shows how $Q'$ and $S'$ would look like in this case. $S'$ separates $Q'$ into two connected components, the closures of which we denote $Q'_0$ and $Q'_1$. Specifically, we denote $Q'_0$ the closure of the ``outer'' component and $Q'_1$ the closure of the ``inner'' one.

\begin{figure}[h]
\null\hfill
\subfigure[The polyhedron $Q'$]{
	\begin{pspicture}(-1,-0.2)(7.2,7)
	\rput[bl](0,0){\scalebox{0.75}{\includegraphics{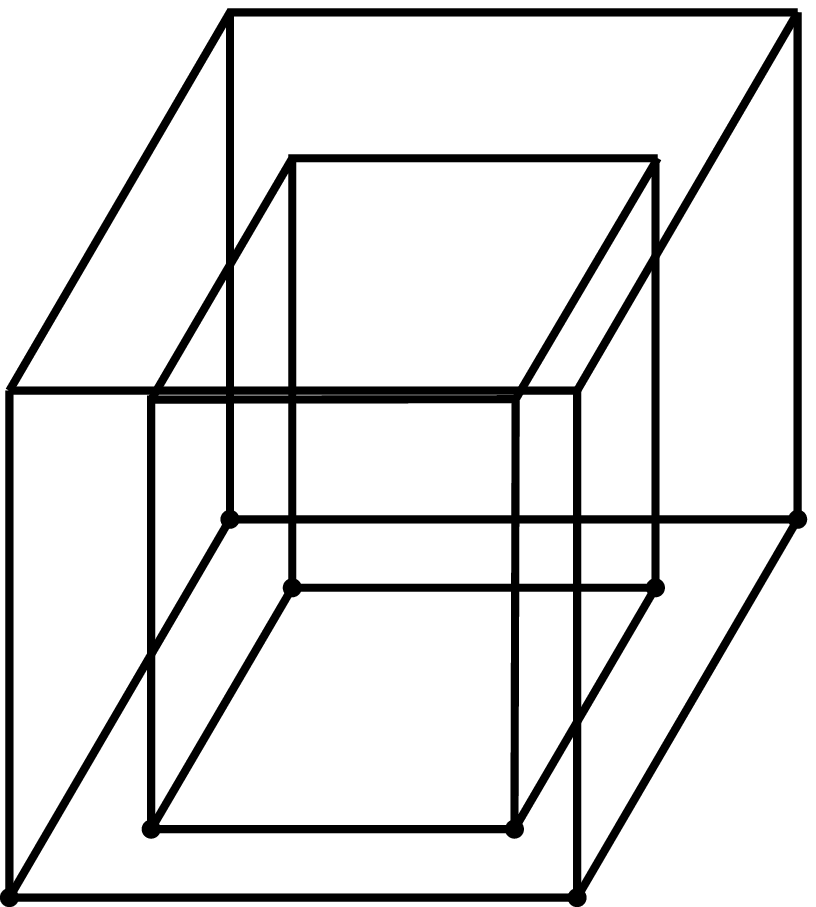}}}
	\rput(0.8,6.2){$Q'$}
	\rput[r](-0.1,0){$A_1$} \rput[l](6.2,3){$A_3$} \rput[l](4.5,0){$A_2$}
	\rput[bl](1.45,0.7){$B_i$}	
	\end{pspicture}}
\hfill
\subfigure[The disk $S'$ contained in $Q'$]{
\begin{pspicture}(0,-0.2)(6.2,7)
	\rput[bl](0,0){\scalebox{0.75}{\includegraphics{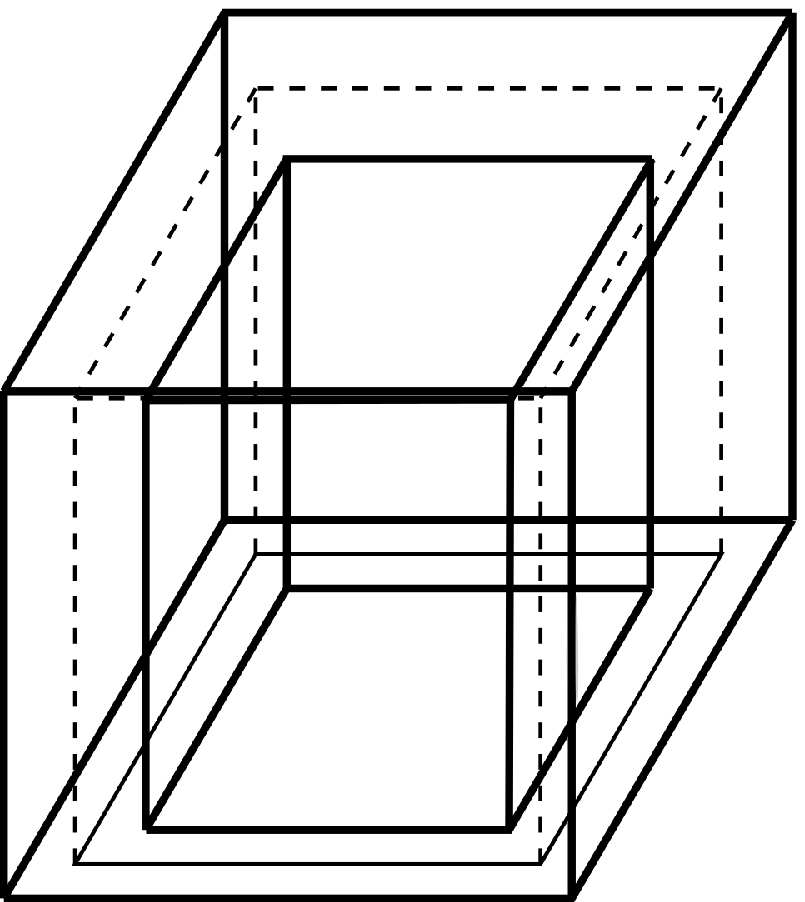}}}
	\rput[br](0.8,6.4){$S'$} \psline{->}(0.8,6.4)(1.3,5.5)
	\end{pspicture}}
\hfill\null
\caption{\label{fig:prime}}
\end{figure}

It is easy to see that there exists a homeomorphism $e$ from $Q$ onto $Q'$ that sends $S$ onto $S'$. A partial description of $e$ will hopefully be enough. Label the vertices of the bottom, annular face of $Q'$ with the letters $A_1, \ldots, A_4$ and $B_1,\ldots,B_4$ as suggested in Figure \ref{fig:prime}.(a) (not all the labels are shown in the drawing to avoid cluttering). Then, referring to Figures \ref{fig:tameaxes} and \ref{fig:prime}.(a), the homeomorphism $e$ may be described by saying that it is piecewise linear and has the following properties:
\begin{itemize}
	\item[({\it i}\/)] The vertices of the left face of $Q$ are mapped onto the $A_i$ while the face itself is mapped onto the outer surface of $Q'$; that is, onto the lateral outer walls and the top outer face.
	\item[({\it ii}\/)] The vertices of the right face of $Q$ are mapped onto the $B_i$; the face itself is mapped onto the inner surface of $Q'$. 
	\item[({\it iii}\/)] A thin annular neighbourhood along $\partial S$ in $\partial Q$ is mapped onto the bottom annular face of $Q'$.
\end{itemize}

Evidently $e$ sends $Q_0$ onto $Q'_0$ and $Q_1$ onto $Q'_1$. Hence $U_0 \cap e(Q) = U_0 \cap Q' \subseteq e(Q_0)$ and similarly $U_1 \cap e(Q) = U_1 \cap Q' \subseteq e(Q_1)$ (the $U_i$ were defined in Equation \eqref{eq:def_U}), showing that $e$ satisfies condition (2). Also, $e(\dot{S})$ is the interior of $e(S) = S'$, so condition (1) is also satisfied because $\hat{D} \subseteq \dot{S}'$ by construction.

\smallskip

{\it General case.} When $D$ is an arbitrary disk we can appeal to the $2$--dimensional Sch\"onflies theorem to reduce the problem to the previous case. Indeed, the Sch\"onflies theorem guarantees that there exists a homeomorphism $s : \{z=0\} \longrightarrow \{z=0\}$ that sends $D$ onto a square $D_0$ for which we already know an embedding $e_0$ exists with the required properties, as constructed in the previous case. Clearly $s$ extends to a homeomorphism $\hat{s}$ of all $\mathbb{R}^3$ simply letting $\hat{s}(x,y,z) := (s(x,y),z)$, and it is then easy to check that the embedding $e := \hat{s} \circ e_0$ satisfies all the required conditions.
\end{proof}

\begin{remark} \label{rem:Qprima} Later on we shall need the explicit form of $Q' = e(Q)$ used in the proof of the proposition. Hence, for future reference, we record it here: \begin{equation} \label{eq:Qprima} Q' := A \times [-1,\nicefrac{3}{2}] \cup D \times [\nicefrac{1}{2},\nicefrac{3}{2}],\end{equation} where $A$ is a closed annulus along the boundary of $D$.
\end{remark}

\subsection{Stage 1} \label{subsec:1} W need to recall some standard definitions (see for instance \cite{brown4}). Suppose $Y$ is a subspace of $X$. We say that $Y$ is \emph{collared} in $X$ if there exists a homeomorphism $c$ from $Y \times [0,1]$ onto a neighbourhood of $Y$ in $X$ and such that $c(p,0) = p$ for every $p \in Y$, and \emph{locally collared} at some point $p \in Y$ if $p$ has a neighbourhood in $Y$ which is collared in $X$. We shall make use of two important theorems of Brown:
\begin{itemize}
	\item[(B1)] If $Y$ is locally collared in $X$ at each $p \in Y$, then $Y$ is collared in $X$ \cite[Theorem 1, p. 337]{brown4}.
	\item[(B2)] As a consequence, the boundary $\partial X$ of a manifold with boundary $X$ is collared in the manifold \cite[Theorem 2, p. 339]{brown4}.
\end{itemize}

Suppose $S$ is a closed surface contained in $\mathbb{R}^3$. A \emph{bicollar} of $S$ is a homeomorphism $b$ from $S \times [-1,1]$ onto a neighbourhood of $S$ in $\mathbb{R}^3$ such that $b(p,0) = p$ for every $p \in S$ (this is a trivial extension of the definition given earlier for $2$--spheres). The existence of a bicollar is equivalent to the existence of a collar of $S$ both in $\bar{U}$ and $\bar{V}$, where $U$ and $V$ are the two connected components into which $S$ separates $\mathbb{R}^3$. When $S$ bounds a $3$--manifold, as with $\partial B$ in our context, then at least one of the complementary domains of $S$ (say $\bar{U}$) is a $3$--manifold whose boundary is precisely $S$, so $S$ is already collared ``on the $U$-side'' and consequently it is bicollared if and only if it is collared also ``on the $V$-side''.

Given a point $p \in S$ on defines the notion of $S$ being \emph{locally bicollared} at $p$ mimicking the definition of locally collared given earlier. It is not difficult to check that if $S$ is locally flat at some point $p$, then it is locally bicollared at that same point (the converse is also true, but we shall not need it).
\smallskip

Let $E$ be a closed disk contained in the boundary of $B$ and such that $\partial E$ is disjoint from $T$ so, in particular, $\partial B$ is locally flat at each point of $\partial E$.

First let us show how to push $E$ into $B$, essentially mimicking the model laid out in Section \ref{sec:stage0}. Since $\partial B$ is collared in $B$ by (B2), there exists an embedding $c^+$ of $\partial B \times [0,2]$ into $B$ such that $c^+(p,0) = p$ for every $p \in \partial B$. (Notice that we take $\partial B \times [0,2]$ rather than $\partial B \times [0,1]$ as the domain of the collar $c^+$; this is just for notational convenience.) Denote $C$ the image of $c^+$, which is a closed neighbourhood of $\partial B$ in $B$. Set \begin{equation} \label{eq:def_E}\hat{E} := c^+(\partial E \times [0,1] \cup E \times 1).\end{equation}  By construction $\hat{E}$ is a $2$--disk properly embedded in $B$; also, $\partial \hat{E} = \partial E$. At this point it may be convenient to refer back to Figure \ref{fig:split_disk}.(b), which shows a two dimensional sketch of the situation. The disk $E$ appears as an interval contained in the boundary of $B$, and $\hat{E}$ can be thought of as the result of pushing $E$ into $B$ along the lines of the collar $c^+$.


\begin{proposition} \label{prop:separates} Suppose $B$ is connected. Then $\hat{E}$ separates $B$ into two connected components $V_0$ and $V_1$ such that $E \cap V_0 = \emptyset$, $\dot{E} \subseteq V_1$ and $\bar{V}_0 \cap \bar{V}_1 = \hat{E}$.
\end{proposition}
\begin{proof} Let $C/2$ denote ``half the collar $C$''; that is, $C/2 := c^+(\partial B \times [0,1])$. Consider the sets \[V_0 := (B - C/2) \cup c^+((\partial B - E) \times [0,2]) \quad \text{and} \quad V_1 := c^+(\dot{E} \times [0,1)).\] Both are open in $B$, and both are connected: $V_1$ clearly so, and $V_0$ because it is the union of two connected sets having a nonempty intersection. Indeed, $B - C/2$ is homeomorphic to the interior of $B$, which is connected because $B$ is connected (in $\mathbb{R}^3$, the assumption that $\partial B$ is connected entails that $B$ is connected too), and $\partial B - E$ is connected because removing a disk from a (by assumption) connected surface does not disconnect the surface. Also, $V_0$ and $V_1$ are disjoint and their union is $B - \hat{E}$, so they are actually the two connected components into which $\hat{E}$ separates $B$. A straightforward computation shows that $\bar{V}_0 = V_0 \cup \hat{E}$ and $\bar{V}_1 = V_1 \cup \hat{E}$, so $\bar{V}_0 \cap \bar{V}_1 = \hat{E}$. By construction clearly $V_0$ does not meet $E$ whereas $\dot{E} \subseteq V_1$.
\end{proof}

Now we want to show that the decomposition $B = \bar{V}_0 \cup \bar{V}_1$ is tame. To this end we need to construct some sort of ``partial bicollar'' of $E$ in $\mathbb{R}^3$ as described below.

The ``inner half'' of the bicollar (that is, the part contained in $B$) presents no difficulty since $\partial B$ itself (hence $E$) is collared in $B$. In fact, we already made use of it when constructing the disk $\hat{E}$. We therefore keep our previous notation: $c^+ : \partial B \times [0,2] \longrightarrow B$ is an embedding such that $c^+(p,0) = p$ for every $p \in \partial B$.

Denote $B'$ the complement (in $\mathbb{R}^3$) of the interior of $B$; that is, $B' := \mathbb{R}^3 - \dot{B}$. The ``outer half'' of the bicollar (the part contained in $B'$) can be constructed along $\partial E$, since it contains no wild points of $\partial B$ by assumption. More precisely, we proceed as follows. At each $p \in \partial E$ the surface $\partial B$ is locally flat, hence locally bicollared and in particular locally collared in $B'$. Thus, every $p \in \partial E$ has a neighbourhood $W_p$ in $\partial B$ over which it is possible to find a local collar $c^-_p : W_p \times [-1,0] \longrightarrow B'$. Since $\partial E$ is compact, we may cover it with a finite family of these $W_p$ and paste all the collars together using the result of Brown cited above (B1) to obtain a collar $c^{-} : W \times [-1,0] \longrightarrow B'$, where $W$ is a neighbourhood of $\partial E$ in $\partial B$.

Notice that $\partial E$ is a simple closed curve and so it has arbitrarily close neighbourhoods $A$ that are closed annuli along the curve $\partial E$. Finding such an $A$ so thin that it is contained in $W$ we can assume that the collar $c^{-}$ is defined on all of $A$. The union $b := c^+ \cup c^-$ thus provides an embedding \[b : A \times [-1,2] \cup E \times [0,2] \longrightarrow \mathbb{R}^3.\]

\begin{proposition} \label{prop:tame_dec} The decomposition $B = \bar{V}_0 \cup \bar{V}_1$ is tame. Thus, the inequality $r(B) \geq r(\bar{V}_0) + r(\bar{V}_1)$ holds.
\end{proposition}
\begin{proof} We are going to use the criterion of Proposition \ref{prop:tame} (thus $Q$, $S$, $Q_0$ and $Q_1$ have the meaning given there).

By the two dimensional Sch\"onflies theorem, there exists a homeomorphism $g_0$ of the plane $\{z = 0\}$ onto an open neighbourhood of $E \cup A$ in $\partial B$. Let $D$ and $A_0$ be the pullbacks of $E$ and $A$ via $g_0$; that is, $D := g_0^{-1}(E)$ and $A_0 := g_0^{-1}(A)$. These are a closed disk in $\{z=0\}$ and an annulus along its boundary, respectively.

As in Stage 0, push $D$ into $P^+$ to depth $1$ obtaining $\hat{D} := \partial D \times [0,1] \cup D \times 1$. Also, consider the cube $Q' := A_0 \times [-1,2] \cup D \times [0,2]$ (compare with Equation \eqref{eq:Qprima}). The map $g_0$ can be extended to an embedding $g$ of all of $Q'$ into $\mathbb{R}^3$ using the partial bicollar $b$, letting $g(x,y,z) := b(g_0(x,y),z)$. Notice that, by construction:
\begin{itemize}
	\item[({\it i}\/)] $g$ takes $\hat{D}$ onto $\hat{E}$,
	\item[({\it ii}\/)] $g(U_0 \cap Q') \subseteq V_0$ and $g(U_1 \cap Q') = V_1$, where $V_0$ and $V_1$ are the connected components of $B - \hat{E}$.
\end{itemize}

According to Proposition \ref{prop:step0} (see also Remark \ref{rem:Qprima}) there exists a homeomorphism $e : Q \longrightarrow Q'$ such that:
\begin{enumerate}
	\item[(1)] $\hat{D} \subseteq e(\dot{S})$,
	\item[(2)] $U_0 \cap e(Q) \subseteq e(Q_0)$ and $U_1 \cap e(Q) \subseteq e(Q_1)$.
\end{enumerate}

Consider the embedding $e' := g \circ e : Q \longrightarrow \mathbb{R}^3$. Applying $g$ to both sides of (1) and using ({\it i}\/) we see that $\hat{E} \subseteq e'(\dot{S})$. Doing the same with (2) and ({\it ii}\/) it follows that $V_0 \cap e'(Q) \subseteq e'(Q_0)$ and $V_1 \cap e'(Q) \subseteq e'(Q_1)$. This establishes (1) and (2) of Proposition \ref{prop:tame}, proving that the decomposition $B = \bar{V}_0 \cup \bar{V}_1$ is tame.


\end{proof}

\begin{remark} \label{rem:depth} Notice that in constructing $\hat{E}$ there is nothing special in pushing $E$ into $B$ to ``depth $1$''; that is, we could set \[\hat{E} := c^+(\partial D \times [0,\epsilon] \cup E \times \epsilon)\] for any $0 < \epsilon < 1$ and everything would work just as well (with the appropriate modifications to the definitions of $D$, $S'$ and $Q'$). For later reference we shall call this ``pushing $E$ into $B$ down to depth $\epsilon$''. Also, observe that the collar $c^+$ can be fixed once and for all; it does not depend on the disk $E$ being considered. The same is not true of $c^{-}$, since it was constructed pasting local collars along $\partial E$.
\end{remark}

\subsection{Stage 2} \label{subsec:2} Let now $E_1,E_2, \ldots,E_n$ be disjoint closed disks in $\partial B$ whose boundaries do not meet $T$. Push each $E_j$ into $B$ while leaving its boundary untouched. This is done exactly as described in \ref{subsec:1} for a single disk: having chosen a collar $c^+ : \partial B \times [0,2] \longrightarrow B$, we consider the family of disjoint disks \[\hat{E}_j := c^+(\partial E_j \times [0,1] \cup E_j \times 1);\] all of them properly embedded in $B$. Arguing as in Proposition \ref{prop:separates} one sees that the (union of the) disks $\hat{E}_1, \hat{E}_2,\ldots,\hat{E}_n$ separate $B$ into $n+1$ connected components $V_0, V_1, \ldots, V_n$. We label them in such a way that $V_0$ does not meet any of the $E_j$ whereas $\dot{E}_j \subseteq V_j$ for each $1 \leq j \leq n$. More explicitly \[V_0 := (B - C/2) \cup c^+((\partial B - \cup_j E_j) \times [0,2]) \quad \text{and} \quad V_j := c^+(\dot{E}_j \times [0,1)) \quad \text{for } j = 1,2,\ldots,n;\] here $C/2$ is the ``half collar'' introduced in the proof of Proposition \ref{prop:separates}. The $V_j$ satisfy the following properties:
\begin{itemize}
	\item[(V1)] $\dot{E}_j \subseteq V_j$ for each $j = 1,\ldots,n$
	\item[(V2)] $\bar{V}_i \cap \bar{V}_j = \emptyset$ for $1 \leq i \neq j \leq n$
	\item[(V3)] $\bar{V}_0 \cap \bar{V}_j = \hat{E}_j$ for each $j = 1,\ldots,n$
\end{itemize}

\begin{proposition} \label{prop:tame_dec_more} The inequality $r(B) \geq \sum_{j=0}^n r(\bar{V}_j)$ holds.
\end{proposition}
\begin{proof} Consider the ascending sequence of compact sets \[K_j := \bar{V}_0 \cup \bar{V}_1 \cup \ldots \cup \bar{V}_j\] for $0 \leq j \leq n$. Notice that $K_{j+1} := K_j \cup \bar{V}_{j+1}$ and $K_j \cap \bar{V}_{j+1} = \bar{V}_0 \cap \bar{V}_{j+1} = \hat{E}_{j+1}$ by (V2) and (V3). The same argument of Proposition \ref{prop:tame_dec} (enlarging the disk $\hat{E}_{j+1}$ using the fact that its boundary does not contain any wild points) shows that the decomposition $K_{j+1} = K_j \cup \bar{V}_{j+1}$ is tame, and so entails that $r(K_{j+1}) \geq r(K_j) + r(\bar{V}_{j+1})$. Inductively, this gives \[r(K_n) \geq r(K_0) + \sum_{j=1}^n r(\bar{V}_j).\]

It only remains to observe that $K_0 = \bar{V}_0$ (this is clear) and $K_n = \bar{V}_0 \cup \bar{V}_1 \cup \ldots \cup \bar{V}_n$ is the whole $B$: indeed, by definition the union of the $V_j$ is $B - \cup \hat{E}_j$; since (V3) guarantees that $\bar{V}_0$ contains all the $\hat{E}_j$, it follows that the union of the $\bar{V}_j$ is all of $B$.
\end{proof}

\subsection{Stage 3} We are finally ready to put all the pieces together and prove Theorem \ref{teo:top}:

\begin{proof}[Proof of Theorem \ref{teo:top}] From the fact that $T$ is totally disconnected it is easy to prove (or see \cite[Theorem 6, p. 72 and Theorem 5, p. 93]{moise2}) that it has arbitrarily small neighbourhoods in $\partial B$ that are unions of a finite number of disjoint disks. Thus for $k = 0,1,2,\ldots$ we may construct a sequence of neighbourhoods $E^{(k)}$ of $T$ in $\partial B$ such that
\begin{itemize}
	\item[(E1)] each $E^{(k)}$ is a union of disjoint disks $E_1^{(k)},E_2^{(k)}, \ldots, E_{n_k}^{(k)}$,
	\item[(E2)] $E^{(k+1)}$ is contained in the interior of $E^{(k)}$ for every $k$,
	\item[(E3)] the intersection $\bigcap_k E^{(k)}$ is precisely $T$.
\end{itemize}

For each $k$ apply the construction described in \ref{subsec:2} to the family of disks $E_1^{(k)}, E_2^{(k)}, \ldots, E_{n_k}^{(k)}$. This involved a choice of a collar $c^+$ of $\partial B$ in $B$; this is to be made once and used for every $k$ as mentioned in Remark \ref{rem:depth}. Also (again, refer to Remark \ref{rem:depth} we can push the disks $E_j^{(k)}$ into $B$ only to depth $\epsilon = 1/k$. The resulting disks $\hat{E}_1^{(k)}, \hat{E}_2^{(k)},\ldots,\hat{E}_{n_k}^{(k)}$ separate $B$ into $n_k+1$ connected components $V_0^{(k)},V_1^{(k)},\ldots,V_{n_k}^{(k)}$ of which the ones of interest to us are \[V_j^{(k)} = c^+(\dot{E}_j^{(k)} \times [0,1/k)) \quad j = 1,2,\ldots,n_k.\] Condition (E2) on the $E^{(k)}$ implies that each disk $E_j^{(k+1)}$ is contained in some disk $E_{j'}^{(k)}$. Since $c^+$ is independent of $k$, this entails that each $V_j^{(k+1)}$ is contained in some $V_{j'}^{(k)}$. Set \[B_0^{(k)} := \bar{V}_0^{(k)} \quad \text{and} \quad B_1^{(k)} := \bigcup_{j=1}^{n_k} \bar{V}_j^{(k)}.\] From what we just said, $B_1^{(k+1)} \subseteq B_1^{(k)}$ for each $k$. Also, (E3) and the condition that the disks $\hat{E}_j^{(k)}$ are obtained by pushing $E_j^{(k)}$ only to depth $1/k$ implies that $\bigcap_k B_1^{(k)} = T$. Finally, Proposition \ref{prop:tame_dec_more} shows that \[r(B) \geq \sum_{j=0}^{n_k} r(\bar{V}_j^{(k)}) \geq \sum_{j=1}^{n_k} r(\bar{V}_j^{(k)}) = r(B_1^{(k)})\] where in the last step we have used the trivial fact that $r$ of a disjoint union of finitely many compact sets is the sum of their $r$ numbers. Now from the semicontinuity of $r$ and the fact that the $B_1^{(k)}$ decrease and have $T$ as their intersection it follows that \[r(B) \geq \liminf_{k \rightarrow \infty} r(B_1^{(k)}) \geq r(T).\] This concludes the proof of Theorem \ref{teo:top}.
\end{proof}

\section{The proof of Theorem \ref{teo:main2} and Corollary \ref{cor:uncountable}} \label{sec:proof_main}

All the work done so far makes the proof of the following result very short:

\begin{theorem} \label{teo:B} Let $B \subseteq \mathbb{R}^3$ be a compact $3$--manifold with a connected boundary $\partial B$. Notice that $\partial B$ is a closed surface. Suppose that $\partial B$ contains a compact, totally disconnected set $T$ such that:
\begin{itemize}
	\item[(i)] $T$ is not rectifiable.
	\item[(ii)] $\partial B$ is locally flat at each $p \not \in T$.
\end{itemize}
Then $B$ cannot be realized as an attractor for a homeomorphism of $\mathbb{R}^3$.
\end{theorem}
\begin{proof} By Theorems \ref{teo:top} and \ref{teo:tot_disc} we have $r(B) \geq r(T) = \infty$, so $r(B) = \infty$. Then the finiteness property (P1) of $r$ (see p. \pageref{P1}) shows that $B$ cannot be an attractor for a homeomorphism.
\end{proof}

From this,

\begin{proof}[Proof of Theorem \ref{teo:main2}] Suppose that the closed and connected surface $S$ were an attractor for a homeomorphism $f$ of $\mathbb{R}^3$. Think of $\mathbb{S}^3$ as $\mathbb{R}^3$ together with the point at infinity $\infty$ and extend $f$ to all of $\mathbb{S}^3$ letting $f(\infty) := \infty$. Viewing $S$ as a subset of $\mathbb{S}^3$, it is still an attractor for $f$ with the same basin of attraction.

Let $U_1$ and $U_2$ be the connected components of $\mathbb{S}^3 - S$. Since $S$ bounds a $3$--manifold by hypothesis, the closure (in $\mathbb{R}^3$ and also in $\mathbb{S}^3$) of at least one of the $\bar{U}_i$ is a $3$--manifold. Call it $B$. Since $S$ is invariant under $f$, so is $\mathbb{S}^3 - S$ and therefore $f$ either leaves each $U_i$ invariant or interchanges them. However, only one of the $U_i$ contains $\infty$, which is fixed by $f$, so it must be the case that $f(U_i) = U_i$. As a consequence $f(B) = B$, so that $B$ is also invariant under $f$. It is straightforward to check that $B$ is an attractor for $f$. Also, by definition $B$ is a compact $3$--manifold whose boundary is precisely $\partial B = S$. Applying Theorem \ref{teo:B} to $B$ we see that $T$ has to be rectifiable, and so Theorem \ref{teo:main2} follows.
\end{proof}

The alternative form of Theorem \ref{teo:main2prima} is an immediate consequence of this:

\begin{proof}[Proof of Theorem \ref{teo:main2prima}] Being a compact manifold, $S$ has finitely many connected components $S_i$ each of which is a closed surface. Also, all the $S_i$ bound a $3$--manifold on one side because $S$ does. Since each $S_i$ is open in $S$, the surface $S$ is locally flat at some $p \in S_i$ if and only if $S_i$ is locally flat at $p$. Therefore the set of wild points $W$ of $S$ is the union of the sets $W_i$ of wild points of the $S_i$. Finally, it is proved in \cite{mio4} that each component of an attractor in $\mathbb{R}^n$ is an attractor itself. Thus all the $S_i$ are attractors, and applying the contrapositive of Theorem \ref{teo:main2} to them (with $T = W_i$ in each case) it follows that each $W_i$ must be rectifiable. Since the $W_i$ are open in $W$ and cover it, $W$ is locally rectifiable and, by Lemma \ref{lem:dico}, rectifiable.
\end{proof}

We saw in Remark \ref{rem:r0tame} that a rectifiable totally disconnected set has a simply connected complement in $\mathbb{R}^3$. Thus we may state the following, more computational consequence of Theorem \ref{teo:main2}:

\begin{corollary} Let $S \subseteq \mathbb{R}^3$ be a closed surface that bounds a $3$--manifold. Suppose that $S$ contains a closed, totally disconnected set $T$ such that (i) $S$ is locally flat at each $p \not\in T$ and (ii) the complement of $T$ in $\mathbb{R}^3$ is not simply connected. Then $S$ cannot be an attractor.
\end{corollary}

We finish by proving that there are uncountably many non equivalent ways of embedding a $2$--sphere in $\mathbb{R}^3$ in such a way that it cannot be an attractor (Corollary \ref{cor:uncountable}). It will be evident from the proof that the corollary is true for (orientable) surfaces of any genus. However, for surfaces of higher genus (think, for instance, of a torus) one can obtain trivially nonequivalent embeddings by knotting the surface differently in $\mathbb{R}^3$. Concentrating on spheres we avoid this degree of freedom, which is not interesting in the present context.

Let us review the construction of nonattracting surfaces described in Section \ref{sec:statement}. We start with a compact $3$--manifold $B_0$, which in our present case is going to be the unit closed $3$--ball, and extract feelers that branch in such a way as to (in the limit) reach a prescribed non rectifiable, compact, totally disconnected set $T$. The branching process is guided by a sequence of neighbourhoods $N_k$ of $T$ chosen in advance. The resulting compact $3$--manifold $B$ is still homeomorphic to the original $B_0$, and its boundary $S$ (which is, accordingly, homeomorphic to $S_0$) is locally flat at each $p \not \in T$. Then Theorem \ref{teo:main2} immediately implies that $S$ cannot be an attractor.

Although for our purposes in Section \ref{sec:statement} we did not need to fully identify the set $W$ of wild points of $S$ (it was enough to know that $W \subseteq T$), now we need to be a little more precise and guarantee that $W = T$. In general this equality does not need to hold: for instance, if $T$ has an isolated point $p$ and its corresponding feeler converges ``straight'' towards $p$ (think of a cone having $p$ at its tip) then $p$ is not a wild point.

There are several ways to refine the above construction in order to guarantee that $W = T$. One of them is to require that $T$ be \emph{homogeneous}; that is, for any two $p, q \in T$ there exists a homeomorphism of $\mathbb{R}^3$ that sends $p$ onto $q$. It is not difficult to prove that: (i) if $S$ is locally flat at some $p \in T$, then $T$ is locally rectifiable at $p$ (essentially because every compact, totally disconnected set in the plane is rectifiable; see Chapter 13 in \cite{moise2}); (ii) due to the homogeneity of $T$, if it is locally rectifiable at a single point then it is locally rectifiable at every point. Since $T$ is non rectifiable, (i), (ii) and Lemma \ref{lem:dico} entail that $W = T$.

A less simple but maybe more illustrative way of achieving the equality $W = T$ consists in patterning the feelers that are used to enlarge $B_0$ after the wild arc $\alpha$ shown in Figure \ref{fig:fox_artin}.(a), due to Fox and Artin.  By ``patterning the feeler along $\alpha$'' we mean that the feeler is obtained by thickening $\alpha$ to a solid tube whose diameter tapers towards zero as it approaches $p$, as shown in Figure \ref{fig:fox_artin}.(b). An arc is tame if it can be sent into a straight line segment by an ambient homeomorphism, or wild otherwise. That $\alpha$ is wild was established by Fox and Artin \cite[Example 1.2, pp. 983 ff.]{foxartin1} performing a local analysis of the fundamental group of $\mathbb{R}^3 - \alpha$ near $p$. More specifically, they showed that $\alpha$ is not $1$--LCC at $p$, which it would be if it were tame (see \cite[Section 2.8, pp. 75 ff.]{davermanvenema1} for more details). We remark, because it will play a role later on, that an arc $\beta \subseteq \mathbb{R}^3$ that is actually contained in a plane within $\mathbb{R}^3$ is $1$--LCC at each of its points because there are no wild arcs in two dimensions (\cite[Chapter 10, pp. 59 ff.]{moise2}).

\begin{figure}[h]
\null\hfill
\subfigure[A Fox--Artin arc]{
\begin{pspicture}(0,0)(3.8,1.8)
\rput[bl](0,0){\scalebox{0.75}{\includegraphics{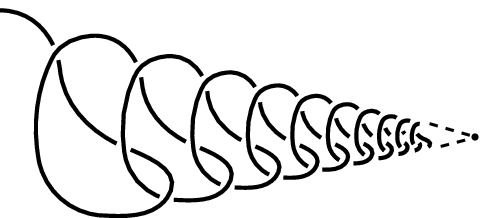}}}
\rput(3.6,0.4){$p$} \rput[bl](0.8,1.6){$\alpha$}
\end{pspicture}}
\hfill
\subfigure[A Fox--Artin feeler]{
\begin{pspicture}(0,0)(3.8,1.8)
\rput[bl](0,0){\scalebox{0.75}{\includegraphics{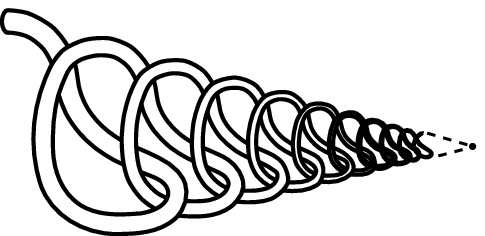}}}
\rput(3.6,0.4){$p$}
\end{pspicture}}
\hfill\null
\caption{\label{fig:fox_artin}}
\end{figure}

In the actual construction of $S$ we need to branch feelers into smaller subfeelers, of course, but we want to preserve their ``Fox--Artin nature''. Figure \ref{fig:feeler} suggests how this is done. As illustrated in Figure \ref{fig:feeler}.(a) suppose that, just after crossing the boundary of some $N_k$ (recall that these are the neighbourhoods of $T$ that are used to guide the branching process), we need to split the incoming feeler into, say (for simplicity), two subfeelers that will eventually converge to two different points $p$ and $p'$ of $T$. We continue as in Figure \ref{fig:feeler}.(b), simply patterning each subfeeler after the Fox--Artin arc $\alpha$. However, this is not all: we insist that the subfeelers are tangled with the original feeler in the characteristic fashion of $\alpha$, as shown for the lower subfeeler in Figure \ref{fig:feeler}.(c). The upper subfeeler (the one leading to $p'$) should also be tangled with its parent feeler in the same way, although this is not shown in the drawing to avoid cluttering it.

\begin{figure}[h]
\null\hfill
\subfigure[]{
\begin{pspicture}(-0.1,0)(3,2)
\rput[bl](0,0){\scalebox{0.75}{\includegraphics{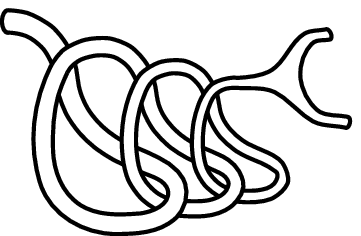}}}
\end{pspicture}}
\hfill
\subfigure[]{
\begin{pspicture}(-0.1,0)(4.6,2)
\rput[bl](0,0){\scalebox{0.75}{\includegraphics{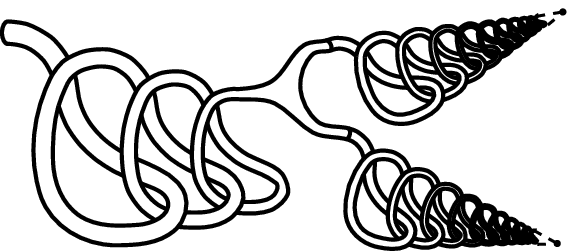}}}
\rput[l](4.4,0){$p$}
\rput[l](4.4,1.9){$p'$}
\end{pspicture}}
\hfill
\subfigure[]{
\begin{pspicture}(-0.1,0)(4.6,2)
\rput[bl](0,0){\scalebox{0.75}{\includegraphics{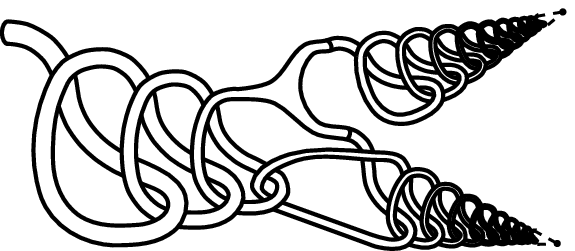}}}
\rput[l](4.4,0){$p$}
\rput[l](4.4,1.9){$p'$}
\end{pspicture}}
\hfill\null
\caption{Branching a Fox--Artin feeler \label{fig:feeler}}
\end{figure}

This last step guarantees that, for each $p \in T$, there is a Fox--Artin arc $\alpha_p$ which looks exactly like $\alpha$ (in particular, it is also a wild arc) that starts at any prescribed point in $B_0$ and ends at $p$, running along the boundary of $B$ (that is, along the surface $S$) and choosing at each branching point the appropriate subfeeler to reach $p$. In turn, this entails that $S$ cannot be locally flat at $p$. For, suppose it were. Then locally near $p$ the surface $S$ would look like a plane in $\mathbb{R}^3$, and $\alpha_p$ (again, near $p$) would be contained in that plane. But we mentioned earlier that there are no wild curves in two dimensions and, in particular, $\alpha_p$ would be $1$--LCC at $p$, which it is not. Thus, summing up, we have proved the following:

\begin{remark} \label{rem:wildset} If the construction of Section \ref{sec:statement} is performed using Fox--Artin feelers as just described, the set of wild points of the resulting surface is all of $T$.
\end{remark}

(Incidentally, maybe after seeing this argument Remark \ref{rem:aboutmain}.(1) becomes more convincing). With Remark \ref{rem:wildset} and a result of Sher the proof of Corollary \ref{cor:uncountable} is now very easy:

\begin{proof}[Proof of Corollary \ref{cor:uncountable}] Sher proved \cite[Corollary 1, p. 1199]{sher1} that there exists an uncountable family $\{T_i\}_{i \in I}$ of Cantor sets in $\mathbb{R}^3$ that are all inequivalently embedded; that is, for any two different $i,j \in I$ there is no ambient homeomorphism sending $T_i$ onto $T_j$. Notice that, since any two rectifiable Cantor sets in $\mathbb{R}^3$ are equivalently embedded (because both can be sent into a straight line, and they are certainly equivalently embedded there by a homeomorphism that can evidently be extended to all of $\mathbb{R}^3$), all the $T_i$ but at most one are non rectifiable. Discarding the latter we may assume without loss of generality that all the $T_i$ are non rectifiable.

For each $T_i$ let $S_i$ be the attracting $2$--sphere obtained by applying the construction of Section \ref{sec:statement}, but using Fox--Artin feelers as just described. This guarantees that the set $W_i$ of wild points of $S_i$ is precisely $T_i$. Also, since the $T_i$ are not rectifiable, an application of Theorem \ref{teo:main2} shows that no $S_i$ can be an attractor.

Now suppose that there exists an ambient homeomorphism $h$ that sends $S_i$ onto $S_j$, with $i \neq j$. It follows immediately from the definition of local flatness that if $S_i$ is locally flat at $p$ then $S_j$ is locally flat at $h(p)$ and conversely or, equivalently stated in terms of wild points, $h(W_i) = W_j$. However, the surfaces were constructed in such a way that $W_i = T_i$ and $W_j = T_j$, so we would have a homeomorphism of $\mathbb{R}^3$ that sends $T_i$ onto $T_j$, thus contradicting the fact that the $T_i$ are all inequivalently embedded.
\end{proof}

\begin{remark} Garity, Repov\v{s} and \v{Z}eljko \cite[Theorem 2, p. 294]{garityrepovszeljko1} strenghtened the result of Sher by showing that there exists an uncountable family of Cantor sets $\{T_i : i \in I\}$ that, in addition to being inequivalently embedded, are all homogeneous. Using this and the argument sketched earlier that the homogeneity of $T_i$ is enough to guarantee that $W_i = T_i$ it is possible to give an alternative, shorter proof of Corollary \ref{cor:uncountable}. We prefer the approach taken above because it avoids the homogeneity condition which, as we have seen, is not essential for the validity of Corollary \ref{cor:uncountable}.
\end{remark}

\bibliographystyle{plain}
\bibliography{biblio}

\begin{thebibliography}{10}

\bibitem{antoine1}
L.~Antoine.
\newblock {Sur l'homéomorphie de figures et de leurs voisinages}.
\newblock {\em J. Math. Pures Appl.}, 86:211--325, 1921.

\bibitem{bing4}
R.~H. Bing.
\newblock {Tame Cantor sets in $E^3$}.
\newblock {\em Pacific J. Math}, 11:435--446, 1961.

\bibitem{brown1}
M.~Brown.
\newblock A proof of the generalized {S}choenflies theorem.
\newblock {\em Bull. Amer. Math. Soc.}, 66:74--76, 1960.

\bibitem{brown4}
M.~Brown.
\newblock Locally flat imbeddings of topological manifolds.
\newblock {\em Ann. Math.}, 75(2):331--341, 1962.

\bibitem{crovisierrams1}
S.~Crovisier and M.~Rams.
\newblock {IFS attractors and Cantor sets}.
\newblock {\em Topology Appl.}, 153:1849--1859, 2006.

\bibitem{davermanvenema1}
R.~J. Daverman and G.~A. Venema.
\newblock {\em Embeddings in manifolds}.
\newblock American Mathematical Society, 2009.

\bibitem{duvallhusch1}
P.~F. Duvall and L.~S. Husch.
\newblock Attractors of iterated function systems.
\newblock {\em Proc. Amer. Math. Soc.}, 116:279--284, 1992.

\bibitem{foxartin1}
R.~H. Fox and E.~Artin.
\newblock Some wild cells and spheres in three-dimensional space.
\newblock {\em Ann. Math. (2)}, 49(4):979--990, 1948.

\bibitem{garay1}
B.~M. Garay.
\newblock Strong cellularity and global asymptotic stability.
\newblock {\em Fund. Math.}, 138:147--154, 1991.

\bibitem{garityrepovszeljko1}
D.~Garity, D.~Repov{\v{s}}, and M.~{\v{Z}}eljko.
\newblock {Uncountably many inequivalent Lipschitz homogeneous Cantor sets in
  $\mathbb{R}^3$}.
\newblock {\em Pacific J. Math.}, 222(2):287--299, 2005.

\bibitem{gunther1}
B.~G{\"u}nther.
\newblock A compactum that cannot be an attractor of a self-map on a manifold.
\newblock {\em Proc. Amer. Math. Soc.}, 120(2):653--655, 1994.

\bibitem{gunthersegal1}
B.~G{\"u}nther and J.~Segal.
\newblock Every attractor of a flow on a manifold has the shape of a finite
  polyhedron.
\newblock {\em Proc. Amer. Math. Soc.}, 119(1):321--329, 1993.

\bibitem{jimenezllibre1}
V.~Jim{\'e}nez and J.~Llibre.
\newblock A topological characterization of the $\omega$--limit sets for
  analytic flows on the plane, the sphere and the projective plane.
\newblock {\em Adv. Math.}, 216:677--710, 2007.

\bibitem{peraltajimenez1}
V.~Jim{\'e}nez and D.~Peralta-Salas.
\newblock Global attractors of analytic plane flows.
\newblock {\em Ergodic Theory Dynam. Systems}, 29:967--981, 2009.

\bibitem{kato1}
H.~Kato.
\newblock {Attractors in Euclidean spaces and shift maps on polyhedra}.
\newblock {\em Houston J. Math.}, 24:671--680, 1998.

\bibitem{moise2}
E.~E. Moise.
\newblock {\em Geometric topology in dimensions 2 and 3}.
\newblock Springer-Verlag, 1977.

\bibitem{munkres3}
J.~Munkres.
\newblock {\em Elements of algebraic topology}.
\newblock Addison--Wesley Publishing Company, Inc., 1984.

\bibitem{ortegayo1}
R.~Ortega and J.~J. S{\'a}nchez-Gabites.
\newblock A homotopical property of attractors.
\newblock {\em Topol. Methods Nonlinear Anal.}, 46:1089--1106, 2015.

\bibitem{rourkesanderson1}
C.~P. Rourke and B.~J. Sanderson.
\newblock {\em Introduction to piecewise-linear topology}.
\newblock Ergebnisse der Mathematik und ihrer Grenzgebiete, Band 69.
  Springer-Verlag, 1972.

\bibitem{mio6}
J.~J. S{\'a}nchez-Gabites.
\newblock Arcs, balls and spheres that cannot be attractors in $\mathbb{R}^3$.
\newblock {\em Trans. Amer. Math. Soc.}
\newblock To appear. Provisionally available at
  \url{http://arxiv.org/abs/1406.5482}.

\bibitem{mio4}
J.~J. S{\'{a}}nchez-Gabites.
\newblock On the shape of attractors for discrete dynamical systems.
\newblock Submitted. Provisionally available at
  \url{http://arxiv.org/abs/1511.06549}.

\bibitem{mio5}
J.~J. S\'anchez-Gabites.
\newblock How strange can an attractor for a dynamical system in a
  $3$--manifold look?
\newblock {\em Nonlinear Anal.}, 74:6162--6185, 2011.

\bibitem{sanjurjo1}
J.~M.~R. Sanjurjo.
\newblock Multihomotopy, \v{C}ech spaces of loops and shape groups.
\newblock {\em Proc. London Math. Soc. (3)}, 69(2):330--344, 1994.

\bibitem{sher1}
R.~B. Sher.
\newblock {Concerning wild Cantor sets in $E^3$}.
\newblock {\em Proc. Amer. Math. Soc.}, 19:1195–1200, 1968.

\end{thebibliography}

\end{document}